\documentclass[11pt,hyp,]{nyjm}
\usepackage{hyperref}
\usepackage{graphicx,amssymb,upref}
\usepackage{frcursive}
\usepackage{mathrsfs}

\title[Test of Vandiver's conjecture with Gauss sums]
{Test of Vandiver's conjecture with  \\ Gauss sums -- Heuristics} 
\author{Georges Gras}
\address{Villa la Gardette, 4 Chemin Ch\^ateau Gagni\`ere, 
F--38520, Le Bourg d'Oisans.}
\email{g.mn.gras@wanadoo.fr}

\newtheorem{theorem}{Theorem}[section]
\newtheorem{lemma}[theorem]{Lemma}
\newtheorem{corollary}[theorem]{Corollary}

\newtheorem{proposition}[theorem]{Proposition}
\newtheorem{hypothesis}[theorem]{Hypothesis}
\theoremstyle{definition}
\newtheorem{definition}[theorem]{Definition}

\theoremstyle{remark}
\newtheorem{remark}[theorem]{Remark}

\newenvironment{pushright}{\begin{itemize}
\item[\hspace{-12pt}]}{\end{itemize}} 
\newcommand{\order}{\raise1.5pt \hbox{${\scriptscriptstyle \#}$}}
\newcommand{\too}{\relbar\mathrel{\mkern-4mu}\rightarrow}
\newcommand{\ds}{\displaystyle}

\newcommand{\Q}{\mathbb{Q}}
\newcommand{\Z}{\mathbb{Z}}
\newcommand{\F}{\mathbb{F}}
\newcommand{\virg}{\raise 2pt \hbox{,}\,\,}

\newcommand{\Cl}{{\mathcal C}\hskip-2pt{\ell}}
\newcommand{\cl}{c\hskip-1pt{\ell}}
\newcommand{\plus}{\ds\mathop{\raise 2.0pt \hbox{$\bigoplus $}}\limits}
\newcommand{\prd}{\ds\mathop{\raise 2.0pt \hbox{$  \prod   $}}\limits}
\newcommand{\sm}{\ds\mathop{\raise 2.0pt \hbox{$  \sum    $}}\limits}
\newcommand{\ev}{\emptyset}
\newcommand{\ul}{\underline}
\newcommand{\wt}{\widetilde}
\newcommand{\ov}{\overline}

\begin{document} 
 
\date{June 25, 2019}

\begin{abstract}  
The link between Vandiver's conjecture and Gauss sums is 
well known since the papers of Iwasawa (1975), 
Thaine (1995-1999) and Angl\`es--Nuccio (2010).
This conjecture is required in many subjects and
we shall give such examples of relevant references.
In this paper, we recall our interpretation of Vandiver's conjecture in terms of 
minus part of the torsion of the Galois group of the maximal abelian 
$p$-ramified pro-$p$-extension of the $p$th cyclotomic field (1984).
Then we provide a specific use of Gauss sums of characters of 
order $p$ of $\F_\ell^\times$ and prove new criteria for Vandiver's 
conjecture to hold (Theorem \ref{first}\,(a) using 
both the sets of exponents of $p$-irregularity and of $p$-primarity of suitable 
twists of the Gauss sums, and Theorem \ref{first}\,(b) which does not need
the knowledge of Bernoulli numbers or cyclotomic units).
We propose in \S\,\ref{new} new heuristics showing that any 
counterexample to the conjecture leads to excessive constraints 
modulo $p$ on the above twists as $\ell$ varies and suggests analytical 
approaches to evidence. We perform numerical experiments to strengthen 
our arguments in direction of the very probable truth of Vandiver's conjecture. 
All the calculations are given with their PARI/GP programs.
\end{abstract}

\maketitle

\tableofcontents

\section{Introduction}
Let $K=\Q(\mu_p^{})$ be the field of $p$th roots of unity for a given prime $p>2$
and let $K_+$ be its maximal real subfield. Put $G :={\rm Gal\,}(K/\Q)$. 

\smallskip
We denote by $\Cl$ and 
$\Cl_+$ the $p$-class groups of $K$ and $K_+$,
then by $\Cl_-$ the relative $p$-class group, so that $\Cl = \Cl _+ \oplus \Cl_-$.

\smallskip
Let $E$ and $E_+$ be the groups of units of  $K$ and $K_+$;
then $E= E_+ \oplus \mu_p^{}$ (Kummer).

\smallskip
The conjecture of Vandiver (or Kummer--Vandiver) asserts that $\Cl_+$ is trivial.
This statement is equivalent to say that the group of real cyclotomic units of $K$
is of prime to $p$ index in $E_+$ \cite[Theorem 8.14]{Wa}. One may refer to numerical 
tables using this property in \cite{BH,HHO} (verifying the conjecture up to
$2\cdot 10^9$), and to more general results in \cite{T1,T2} where some relations 
with Gauss and Jacobi sums are used to express the order of the isotypic 
components of $\Cl_+$ (e.g., \cite[Theorem 4]{T1}).

\smallskip
Many heuristics are proposed about this conjecture; 
see Washington's book \cite[\S\,8.3, Corollary 8.19]{Wa}
for some history, criteria, and for probabilistic arguments,
then \cite{Mi} assuming Greenberg's conjecture \cite{Gre1}
for $K_+$.

\smallskip
We have also given a probabilistic study in \cite[II.5.4.9.2]{Gr1}. 
All these heuristics lead to the fact that the number of 
primes $p$ less than $x$, giving a counterexample,
can be of the form $O(1)\cdot {\rm log}({\rm log}(x))$.

\smallskip
These reasonings, giving the possible existence of 
infinitely many counterexamples to Vandiver's conjecture, are 
based on standard probabilities associated with the
Borel--Cantelli heuristic, but many recent $p$-adic heuristics and
conjectures (on class groups and units) may contradict 
such unfounded approaches. 

\smallskip
In this paper, we shall work in another direction, in the framework of ``abelian 
$p$-ramifi\-cation'', using Gauss sums together with the ``Main Theorem on 
abelian fields'' restricted to $\Cl_-$, and giving the order of its isotypic 
components by means of generalized Bernoulli numbers
(this aspect is related by Ribet in \cite{R1,R2} and we shall 
call it ``Main Theorem'' for short).

\smallskip
Such a link of Vandiver's conjecture with Gauss sums 
and abelian $p$-ramification has been given first by 
Iwasawa \cite{Iw}, then by Angl\`es--Nuccio \cite{AN}, and encountered by 
many authors in various directions (Iwasawa's theory, 
Galois cohomology, Fermat curves, Galois representations,...), 
then often assuming Vandiver's conjecture 
(e.g., \cite{DP,Gre2, Gre3,Ich,IK,KM,Sha0,Sha1,Sha2,Shu,WE1,WE2}).

\smallskip
This link does exist also in the context of the classical conjecture of 
Greenberg \cite{Gre1} considered as a generalization 
of Vandiver's conjecture (e.g., \cite{McC}, \cite{Gr7}).
We propose, in Section \ref{pram}, to explain the links 
with $p$-ramification and prove again the reflection theorem
(Theorem \ref{reflection0} and Corollary \ref{reflection}).

\smallskip
Then we shall interpret a counterexample to Vandiver's conjecture
in terms of non-trivial ``$p$-primary pseudo-units'' 
stemming from Gauss sums:

\smallskip
\centerline{$\tau(\psi) = -\sm_{x \in \F_\ell^\times} \psi(x)\,\xi_\ell^x$,}

for $\psi$ of order $p$, $\xi_\ell$ of prime order $\ell \equiv 1 \pmod p$.
Indeed, if $\order \Cl_+ \equiv 0 \pmod p$,
there exists a class $\gamma = \cl({\mathfrak A}) \in \Cl_-$, of
order $p$, such that ${\mathfrak A}^p = (\alpha)$, with $\alpha$
$p$-primary (to give the unramified extension 
$K(\sqrt[p]{\alpha})/K$, decomposed over $K_+$ into a cyclic
unramified extension $L_+/K_+$ of degree $p$ predicted by class field 
theory); the reciprocal being obvious.

\smallskip
Since $\alpha$ can be obtained explicitely by means of twists
(giving products of Jacobi sums) of the above Gauss sums:
\begin{equation}\label{twist}
{\rm g}_c(\ell) = \tau(\psi)^{c - \sigma_c} \in K, 
\end{equation} 

with Artin automorphisms $\sigma_c$ attached to a primitive 
root $c$ modulo $p$, 
this will yield the main test verifying the validity of the conjecture 
at $p$; this result is the object of the Theorem \ref{thmp},
Corollary \ref{casvide} and Theorem \ref{N}, 
that we can summarize, in the Theorem \ref{first}
below, after the reminder of some notations and classical definitions.

\smallskip
\begin{definition}\label{definit}
(i) Let $\zeta_p$ be a primitive $p$th root of unity.
 We denote by $\omega$ the Teichm\"uller character of $G$ (the 
 $p$-adic character with values in $\mu_{p-1}^{}(\Q_p)$ such that 
$\zeta_p^s=\zeta_p^{\omega(s)}$ for all $s\in G$).

\smallskip
The irreducible $p$-adic characters of $G$ are the
$\theta = \omega^m$, $1 \leq m \leq p-1$.

\smallskip
(ii) Let $e_\theta :=  \hbox{$\frac{1}{p-1}$} \sm_{s \in G} \theta(s^{-1})\, s$ be
the associated idempotents in $\Z_p[G]$.

\smallskip
(iii) Let ${\rm g}_c(\ell)_ \theta$ denotes the $\theta$-component of the twist 
${\rm g}_c(\ell)$ defined by  \eqref{twist}, as representative in $K^\times$ of the 
class of $\ov {{\rm g}_c(\ell)}^{\, e_\theta^{}} \in K^{\times}/K^{\times p}$.
\end{definition}

\begin{theorem}[Main theorem]\label{first}
For a prime $\ell \equiv 1 \pmod p$, let
${\mathscr E}_\ell(p)$ be the set of exponents of $p$-primarity 
of $\ell$ (even integers $n \in [2, p-3]$, such that 
${\rm g}_c(\ell)_{\omega^{p-n}} \equiv 1 \pmod p$).
Then let ${\mathscr E}_0(p)$ be the set of exponents of 
$p$-irregularity of $K$ (even integers $n \in [2, p-3]$, 
such that $p$ divides the $n$th Bernoulli number $B_n$).

\smallskip
\quad (a) Vandiver's conjecture holds for $K$ if and only if
there exists $\ell\equiv 1\! \pmod p$ such that
${\mathscr E}_\ell(p) \cap {\mathscr E}_0(p)=\ev$.

\smallskip
\quad (b) Vandiver's conjecture holds for $K$ if and only if
there exist $N\geq 1$ primes $\ell_i \equiv 1 \!\pmod p$ such that 
$\bigcap_{i=1}^N{\mathscr E}_{\ell_i}(p)=\ev$. 
\end{theorem}

\medskip
Test (b) is numerically very frequent for $N=1$
or $N$ very small, {\it and does not need the 
knowledge of Bernoulli's numbers}; in fact, it does not need
to know if $p$ is irregular or not (see Theorem \ref{N}).

\smallskip
We show that some assumption of {\it independence}, of the 
congruential properties (mod~$p$) of these twists, 
as $\ell$ varies, is an obstruction to any counterexample 
to Vandiver's conjecture.

\smallskip
This method is different from that needed to prove that some cyclotomic 
unit is not a global $p$th power, which does not give obvious probabilistic 
approach (nevertheless, see \S\,\ref{resymbol} for some complements).

\smallskip
Finally, we propose, in \S\S\,\ref{new}, \ref{add}, new heuristics (to our knowledge) 
and give substantial numerical experiments confirming them.

\begin{definition}\label{ND}
(i) We denote by ${\mathscr X}_+$ the set of even characters
$\theta \ne 1$ (i.e., $\theta = \omega^m$, $m \in [2, p-3]$ even), and by
${\mathscr X}_-$ the set of odd characters distinct from $\omega$
 (i.e., $\theta = \omega^m$, $m \in [3, p-2]$ odd).

\smallskip
If $\theta =\omega^m$, we put $\theta^* := 
\omega \theta^{-1}= \omega^{p-m}$. This defines an involution
on the group of characters which applies ${\mathscr X}_+$ onto
${\mathscr X}_+^* = {\mathscr X}_-$.

\smallskip
(ii) For a finitely generated $\Z_p[G]$-module $M$, we put
$M_\theta := M^{e_\theta}$. The operation of the complex 
conjugation $s_{-1} \in G$ gives rise to the obvious definition of the 
components $M_+$ and $M_-$ such that $M=M_+ \oplus M_-$.

\smallskip
(iii) We denote by ${\rm rk}_p(A) := {\rm dim}_{\F_p}(A/A^p)$ the $p$-rank 
of any abelian group $A$.

\smallskip
(iv) For $\alpha \in K^\times$, prime to $p$, considered 
modulo $K^{\times p}$, we denote by $\alpha_\theta$ 
a representative in $K^{\times}$ of the class $\ov \alpha^{\,e_\theta} \in
K^{\times}/K^{\times p}$ (e.g., $\alpha_\theta = \alpha^{e'_\theta}$ where 
$e'_\theta \in \Z[G]$ approximates $e_\theta$ mod $p$).

\smallskip
(v) Let $I$ be the group of prime to $p$ ideals of $K$.
For any ${\mathfrak A}  \in I$ such that 
$\cl({\mathfrak A}) \in \Cl$, there exists an approximation 
$e'_\theta \in \Z[G]$ of $e_\theta$ modulo a sufficient high power 
of $p$ such that ${\mathfrak A}_\theta := {\mathfrak A}^{e'_\theta}$ 
is defined up to a principal ideal of the form $(x^p)$, $x \in K^\times$.

\smallskip
(vi) We say that ${\mathfrak A} \in I$ is $p$-principal if it is principal
in $I \otimes \Z_p$; thus ${\mathfrak A} = (\alpha)$, with 
$\alpha \in K^\times \otimes \Z_p$, defined up to the product by
$\varepsilon \in E \otimes \Z_p$.\,\footnote{The distinction between
${\mathfrak A}^{e_\theta} \in I \otimes \Z_p$ and ${\mathfrak A}^{e'_\theta} \in I$
($e'_\theta \equiv e_\theta \pmod {p^N\Z_p[G]}$, $N$ large enough)
has some importance in practice and programming, provided of a 
definition of ${\mathfrak A}^{e'_\theta}$ up to a principal ideal of the 
form $(x^p)$, for deciding, for instance in the writing
${\mathfrak A}^{e_\theta} =: (\alpha_\theta^{})$, of the``$p$-primarity'' of 
$\alpha_\theta^{} \in K^\times \otimes \Z_p$;
whence ${\mathfrak A}^{e'_\theta} =: (\alpha' )$ where
$\alpha_\theta^{} \cdot \alpha'{}^{-1} \in (K^\times \otimes \Z_p)^p \cdot E \otimes \Z_p$.
This will be used for $\theta \in {\mathscr X}_+^*$ where $\theta$-components of
units do not intervenne, giving $(\alpha_\theta) = (x)^p \Leftrightarrow
\alpha_\theta \in K^{\times p}$.}

(vii) For $\chi =: \omega^n \in {\mathscr X}_+$, let
$B_{1, \,(\chi^*)^{-1}} = B_{1,\,\omega^{n-1}} := 
\frac{1}{p} \sm_{a=1}^{p-1} ({\chi^*})^{-1}(s_a)\, a$ be
the generalized Bernoulli number of character $ ({\chi^*})^{-1}$
(where $s_a \in G$ is the Artin automorphism attached to $a$; it is the restriction
of the Artin automorphism $\sigma_a$ defined above in larger extensions).

\smallskip
The Bernoulli number $B_{1,\,\omega^{n-1}}$ is an element of $\Z_p$ 
congruent modulo $p$ to $\frac{B_n}{n}$, 
where $B_n$ is the $n$th ordinary Bernoulli number; see 
\cite[Proposition 4.1, Corollary 5.15]{Wa}.

\smallskip
(viii) We say that a finitely generated $\Z_p[G]$-module $M$ is monogenous
if it is generated, over $\Z_p[G]$, by a single element; this is equivalent to 
${\rm rk}_p(M_\theta) \leq 1$ for all irreducible $p$-adic character 
$\theta$ of $G$.
\end{definition}

The index of $p$-irregularity $i(p)$ is the number of even $n \in [2, p-3]$ 
such that $B_n \equiv 0 \pmod p$; thus $i(p) = \order {\mathcal E}_0(p)$.
See \cite[\S\,5.3 \ \&\  Exercise 6.6]{Wa} giving statistics and the heuristic 
$i(p) = O\big(\frac{{\rm log}(p)}{{\rm log}({\rm log}(p))} \big)$.

\smallskip
For a general history of Bernoulli--Kummer--Herbrand--Ribet,
then Mazur--Wiles--Thaine--Kolyvagin--Rubin--Greither works on cyclotomy
see \cite{Gr00,R2,Wa}; in this context, if for $\theta \in {\mathscr X}_-$,
$B_{1, \,\theta^{-1}}$ is of $p$-valuation $e$, we shall have (Main Theorem):

\centerline{$\order \Cl_\theta = \big\vert B_{1, \,\theta^{-1}} \big \vert_p^{-1} = p^e$.}

\section{Pseudo-units -- Notion of $p$-primarity}

\begin{definition}\label{psu}
(i) We call {\it pseudo-unit} any $\alpha \in K^\times$, prime to $p$,
such that $(\alpha)$ is the $p$th power of an ideal of $K$.

\smallskip
(ii) We say that an arbitrary $\alpha \in  K^{\times}$,
prime to $p$, is {\it $p$-primary}
if the Kummer extension $K(\sqrt [p] {\alpha}\,)/K$
is unramified at the unique prime ideal ${\mathfrak p}$ 
above $p$ in $K$ (but possibly ramified elsewhere).
\end{definition}

\begin{remark}
(i) Let $A$ be the group of pseudo-units of $K$. If $\alpha \in A$, there
exists an ideal ${\mathfrak a}$ such that $(\alpha) = {\mathfrak a}^p$;
then if we associate with $\alpha K^{\times p}$ the class of ${\mathfrak a}$,
we obtain the exact sequence, where ${}_p\Cl := \{\gamma \in \Cl,\ \, \gamma ^p=1\}$:
$$1 \too E/E^p \too A K^{\times p}/K^{\times p} \too {}_p\Cl \too 1, $$
giving ${\rm dim}_{\F_p} (A K^{\times p}/K^{\times p}) = \frac{p-1}{2}+ {\rm rk}_p(\Cl)$. 
Thus the computation of
${\rm dim}_{\F_p} \big((A K^{\times p}/K^{\times p})_\theta\big)$
is immediate from the value of ${\rm rk}_p(\Cl_\theta)$ and 
${\rm dim}_{\F_p} (\big(E/E^p)_\theta\big)=1$
(resp. $0$) if $\theta \in {\mathscr X}_+ \cup \{\omega\}$ 
(resp. $\theta \in {\mathscr X}_+^* \cup \{1\}$).

\smallskip
(ii) The general condition of $p$-primarity 
for any $\alpha \in K^\times$ ($\alpha$ prime to $p$ but not necessarily
a pseudo-unit) is ``\,$\alpha$ congruent to a $p$th power
modulo ${\mathfrak p}^{p} = (p) \,{\mathfrak p}$\,''
(e.g., \cite[Ch. I, \S\,6, (b), Theorem 6.3]{Gr1}). Since in
any case (replacing $\alpha$ by $\alpha^{p-1}$)
we can assume $\alpha \equiv 1 \pmod {\mathfrak p}$,
the above condition is then
equivalent to $\alpha \equiv 1 \pmod {{\mathfrak p}^{p}}$
(indeed, for any $x \equiv 1 \pmod {\mathfrak p}$ 
we get $x^p \equiv 1 \pmod {{\mathfrak p}^p}$).
\end{remark}

For the pseudo-units of $K$, the $p$-primarity may be 
characterized as follows:

\begin{proposition} Let $\alpha \in K^\times$ be a pseudo-unit. 
Then $\alpha$ is $p$-primary if and only if it is a local $p$th power at~${\mathfrak p}$.
\end{proposition}

\begin{proof} 
One direction is trivial. Suppose that $K(\sqrt [p]{\alpha}\,)/K$ is 
unramified at ${\mathfrak p}$; since $\alpha = {\mathfrak a}^p$, 
this extension is unramified as a global 
extension and is contained in the $p$-Hilbert class field $H$ of $K$. 
The Frobenius automorphism in $H/K$ of the {\it principal ideal}
${\mathfrak p}=(\zeta_p - 1)$ is trivial; so ${\mathfrak p}$ totally splits 
in $H/K$, thus in $K(\sqrt [p] {\alpha}\,)/K$, proving the proposition.
\end{proof}

When $\alpha$ is not necessarily a pseudo-unit, we have a similar
result provided we only look at the $p$-primarity of 
$\alpha_\theta$ for $\theta \ne 1, \omega$:

\begin{proposition}\label{varpi}
Let $\alpha \equiv 1 \pmod {\mathfrak p}$. 
Let $m \in [2,p-2]$ and let $\alpha_\theta$ for $\theta = \omega^m$. Then 
$\alpha_\theta \equiv 1 \pmod {{\mathfrak p}^m}$; moreover $\alpha_\theta$ 
is $p$-primary if and only if $\alpha_\theta \equiv 1 \pmod p$, in which case
one gets $\alpha_\theta \equiv 1 \pmod {{\mathfrak p}^{m+p-1} = (p) {\mathfrak p}^m}$.
\end{proposition}

\begin{proof}
Consider the Dwork uniformizing parameter $\varpi$ in $\Z_p[\mu_p^{}]$
which has the following properties:

\smallskip
(i) $\varpi^{p-1} = -p$,

\smallskip
(ii) $s (\varpi) = \omega(s) \cdot \varpi$, for all $s \in G$.

\medskip
Put $\alpha_{\theta} =1+ \varpi^k u$, 
where $u$ is a unit of  $\Z_p[\varpi]$ and $k \geq 1$; let
$u_0 \in \Z \setminus  p\,\Z$ be such that $u \equiv u_0 \pmod \varpi$
giving $\alpha_{\theta} \equiv 1 + \varpi^k u_0 \pmod  {\varpi^{k+1}}$.
Since $\alpha_{\theta}^{s} = \alpha_{\theta}^{\theta(s)}$ in 
$K^\times \otimes \Z_p$, we get, for all $s \in G$:
$$1 + s (\varpi^k)\, u_0 = 1 + \omega^k(s) \,\varpi^k u_0 
\equiv (1 + \varpi^k u_0)^{\theta(s)} 
  \equiv 1 +  \omega^m(s) \,\varpi^k u_0  \pmod {\varpi^{k+1}},$$

which implies $k \equiv m \!\pmod {p-1}$ and
$\alpha_{\theta} = 1 + \varpi^k \,u$,
$k \in \{m, m+p-1, \ldots\}$; whence the first claim.
The $p$-primarity condition for $\alpha_\theta$ is 
$\alpha_\theta \equiv 1 \pmod {\varpi^p}$ giving the obvious 
direction ``$\alpha_\theta$ $p$-primary $\Rightarrow$ $\alpha_\theta \equiv 1 \pmod p$''
since $(\varpi^p) = (p\, \varpi)$.

\smallskip
Suppose $\alpha_\theta \equiv 1 \pmod {\varpi^{p-1}}$;
so $k=m$ does not work in the writing $\alpha_{\theta} =1+ \varpi^k u$
since $m \leq p-2$, and necessarily $k$ is at least 
$m+p-1\geq p+1$, because $m\geq 2$ (which is also the local $p$th power condition).
\end{proof}

\section{Abelian $p$-ramification} 
Let's give an overview of the theory of abelian $p$-ramification, 
which is not our main purpose, but the natural framework for 
Vandiver's conjecture and Gauss sums.

\subsection{Vandiver's conjecture and abelian 
$p$-ramification}\label{pram}

Let $U$ be the group of principal local units at $p$ of $K$ and let 
$\ov E$ be the closure of the image of $E$ in~$U$.
Let ${\mathcal T}$ be the torsion group of the Galois group of
the maximal abelian $p$-ramified (i.e., unramified outside $p$)
pro-$p$-extension $H^{\rm pr}$ of $K$. This extension contains
the $p$-Hilbert class field $H$ and the compositum $\wt K$ of the 
$\Z_p$-extensions of $K$.
In the case of $K=\Q(\mu_p)$, the theory is summarized by the 
following exact sequences (since Leopoldt's conjecture 
holds for abelian fields):
\begin{equation*}\label{sef}
\begin{aligned}
&1 \too  {\rm tor}_{\Z_p}^{} \big(U \big / \ov E \big) 
\too {\mathcal T}  \too \wt{\Cl} \too 1 \\
& 1\too  {\rm tor}_{\Z_p}^{} (U) / {\rm tor}_{\Z_p}^{}(\ov E)= 1 \too
 {\rm tor}_{\Z_p}^{} \big(U \big / \ov E \big)  
\mathop {\too}^{{\rm log}} {\mathcal R}  \too 0,
\end{aligned}
\end{equation*}

where $\wt{\Cl} \subseteq \Cl$ corresponds, by class 
field theory, to the subgroup ${\rm Gal (H/H \cap \wt K)}$, and where
${\mathcal R} := {\rm tor}_{\Z_p}^{}\big({\rm log}\big 
(U \big) \big / {\rm log} (\ov E) \big)$ is 
the normalized $p$-adic regulator \cite[Proposition 5.2]{Gr5}.
Taking the $\theta$-components, we obtain the exact 
sequences (where ${\mathcal R}_\theta = 1$ for all odd $\theta$):
$$1 \too {\mathcal R}_\theta \too {\mathcal T}_\theta \too \wt\Cl_\theta \too 1.$$

For more information, see \cite{Gr1,Gr2,Gr5}.
We then have
${\rm Gal}(H^{\rm pr}/K) \simeq \Gamma \oplus {\mathcal T} 
\simeq  \Z_p^{\frac{p+1}{2}} \oplus {\mathcal T}$
where $\Gamma := {\rm Gal}(\wt K/K)$ is such that 
$\Gamma_+ = \Gamma_1 \simeq \Z_p$ and $\Gamma_- \simeq \Z_p[G]_-$ 
giving $\Gamma_\theta \simeq \Z_p$ for all odd $\theta$. 

\smallskip
Write ${\mathcal T} = {\mathcal T}_+ \oplus {\mathcal T}_-$ and
define $H^{\rm pr}_- \subseteq H^{\rm pr}$ 
(fixed by ${\rm Gal}(H^{\rm pr}/K)_+$), then
$H^{\rm pr}_+ \subseteq H^{\rm pr}$
(fixed by ${\rm Gal}(H^{\rm pr}/K)_-$). Thus
${\rm Gal}(H^{\rm pr}_+/K) \simeq \Z_p \oplus {\mathcal T}_+$
and ${\rm Gal}(H^{\rm pr}_-/K) \simeq \Z_p^{\frac{p-1}{2}}
\oplus {\mathcal T}_-$.

\smallskip
One defines in the same way the fields 
$H^{\rm pr}_\theta$ for which ${\rm Gal}(H^{\rm pr}_\theta/K)
\simeq \Gamma_\theta \oplus {\mathcal T}_\theta$
(reduced to ${\mathcal T}_\theta$, finite, for all 
$\theta \in {\mathscr X}_+$). We have $H_\theta \subset 
H^{\rm pr}_\theta$ in terms of components of~$H$.

\smallskip
Note that $H^{\rm pr}_+/K$ is decomposed
over $K_+$ to give the maximal abelian $p$-ramified 
pro-$p$-extension of $K_+$. 

\begin{theorem}\label{reflection0}\label{spiegel}
For all irreducible $p$-adic character $\theta$ of $K$, we have
${\rm rk}_p({\mathcal T}_{\theta^*}) = {\rm  rk}_p(\Cl_\theta)$.
\end{theorem}

\begin{proof}
We will give an outline of this famous reflection result as follows
from classical Kummer duality between radicals and Galois groups 
(see, e.g., \cite[Theorem I.6.2 \& Corollary I.6.2.1]{Gr1}),
using the fact that $K(\sqrt[p]{\beta})/K$, $\beta \in K^\times$,
is $p$-ramified if and only if
$(\beta) = {\mathfrak p}^e \cdot {\mathfrak A}^p$, $e \geq 0$,
${\mathfrak A} \in I$.  We shall have to take the
$\theta$ or $\theta^*$-components for each object considered in
$K^\times \otimes \Z_p$, $I  \otimes \Z_p$\,$\ldots\,$, modulo $p$th powers:

\smallskip
Let $\theta$ be even.
The Kummer radical of the compositum of the cyclic extensions 
of degree $p$ of $K$, contained in $H^{\rm pr}_{\theta^*}$, is generated 
(modulo $K^{\times p}$) by the part $E_{\theta}$ of real units, 
giving a $p$-rank $1$ for $\theta \ne 1$ (and $0$ for $\theta=1$),
by $p$ (of character $1$), and by the pseudo-units $\alpha_\theta$ 
comming from the elements of order $p$ of $\Cl_\theta$, 
which gives a radical of $p$-rank $1+{\rm rk}_p(\Cl_\theta)$. 
Since ${\rm rk}_p({\rm Gal}(H^{\rm pr}_{\theta^*}/K))=
1+{\rm rk}_p({\mathcal T}_{\theta^*})$, we get
${\rm rk}_p({\mathcal T}_{\theta^*}) = {\rm rk}_p(\Cl_\theta)$.
Similarly, we have ${\rm rk}_p({\mathcal T}_{\theta}) = {\rm rk}_p(\Cl_{\theta^*})$.
\end{proof}

\begin{corollary}\label{reflection}
One has ${\mathcal T}_1={\mathcal T}_{\omega} = \Cl_\omega=\Cl_1=1$ and
for all $\chi \in {\mathscr X}_+$, we have ${\mathcal R}_{\chi*}=1$ and
${\mathcal T}_{\chi*} = \wt \Cl_{\chi*} \subseteq \Cl_{\chi*}$, which 
establishes the Hecke reflection theorem
or Leopoldt spiegelungssatz ${\rm  rk}_p(\Cl_{\chi*}) = 
{\rm  rk}_p(\Cl_{\chi}) + \delta_\chi$, $\delta_\chi \in \{0,1\}$
since $\Gamma_{\chi*} \simeq \Z_p$
(particular case of \cite[Theorem II.5.4.5, 5.4.9.2]{Gr1}).
\end{corollary}

\begin{remark}
(i) One says that $K$ is $p$-rational if ${\mathcal T}=1$ (same definition
for any number field fulfilling the Leopoldt conjecture at $p$; 
see \cite{Gr2,Gr8} for more details and programs testing the $p$-rationality
of any number field). 
For the $p$th cyclotomic field $K$ this is equivalent to its 
``$p$-regularity'' in the more general context of ``regular kernel'' given in 
\cite[Th\'eor\`eme 4.1]{GJ} (${\mathcal T}_- = 1$ may be 
interpreted as the conjectural ``relative $p$-rationality'' of $K$).

\smallskip
(ii) As we have seen, at each unramified cyclic extension 
$L_+$ of degree $p$ of $K_+$ is associated a $p$-primary 
pseudo-unit $\alpha \in (K^\times/K^{\times p})_-$ such that
$L_+ K= K(\sqrt[p]{\alpha})$. Put $(\alpha)= {\mathfrak A}^p$, 
where $\cl({\mathfrak A}) \in \Cl_-$; moreover ${\mathfrak A}$ 
is not principal, otherwise $\alpha$ should be, up to a $p$th power 
factor, a unit $\varepsilon$ such that $\varepsilon^{1+s_{-1}} = 1$, 
which gives $\varepsilon \in \mu_p^{}$ (absurd). 
In the same way, if $G$ operates via $\chi$ on ${\rm Gal}(L_+/K_+)$ 
then by Kummer duality $G$ operates via $\chi^*$ on 
$\langle \alpha \rangle K^{\times p}/K^{\times p}$.

\smallskip
(iii) As explained in the Introduction, we shall prove in 
Section \ref{sec4} that such pseudo-units $\alpha$ may be found 
by means of twists ${\rm g}_c(\ell):=\tau(\psi)^{c-\sigma_c}$  
associated to primes $\ell \equiv 1 \!\!\pmod p$ and Artin 
automorphisms $\sigma_c$.
\end{remark}

\subsection{Vandiver's conjecture and Gauss sums}\label{VG}

Recall the formula \cite[Corollary III.2.6.1]{Gr1}:
$$\order {\mathcal T}_- = \frac{\order \Cl_-}
{\order \big(\Z_p \,{\rm log}(I)\, \big /\, \Z_p \,{\rm log}(U)\big)_-}, $$

where $I$ is the group of prime to $p$ ideals of $K$ and
$U =1 + \varpi\,\Z_p[\varpi]$.
For any ${\mathfrak A} \in I$, 
let $m \geq 1$ be such that ${\mathfrak A}^m = (\alpha)$, then
${\rm log}({\mathfrak A}) := \frac{1}{m}{\rm log}(\alpha)$ 
where ${\rm log}$ is the $p$-adic logarithm; taking the minus parts,
${\rm log}({\mathfrak A})$ becomes well-defined since 
$\Q_p {\rm log}(E)_- = 0$. We obtain:
\begin{equation}\label{tauchi}
\order {\mathcal T}_{\chi*} = \frac {\order \Cl_{\chi*}}
{\order \big(\Z_p \,{\rm log}(I) \big / \Z_p \,{\rm log}(U)\big)_{\chi*}},
\ \, \hbox{for all $\chi =: \omega^n \in {\mathscr X}_+$}
\end{equation}

The following reasonning (from \cite[\S\,3]{Gr4}) gives
another interpretation of the result of Iwasawa \cite{Iw}. 
Consider  the Stickelberger element
$S :=\frac{1}{p} \sm_{a=1}^{p-1}  a\, s_a^{-1} \in \Q[G]$; it is such that
$\hbox{$S\,.\, e_{\chi*} = B_{1, \,(\chi^*)^{-1}}\,.\,e_{\chi*} \in \Z_p[G]$ for all 
$\chi \in {\mathscr X}_+$;} $
then $\chi^* = \omega^{p-n}$
for which $\order \Cl_{\chi*}$ corresponds to the ordinary 
Berrnoulli numbers $B_n$ giving the ``exponents of 
$p$-irregularity'' $n$ for $B_n \equiv 0 \pmod p$
(see Definitions \ref{ND} (vii)).

\smallskip
Let $\ell$ be a prime number totally split in $K$ 
(thus $\ell \equiv 1 \pmod p$). Let $\psi$ be a 
character of order $p$ of $\F_\ell^\times$.
We define the Gauss sum (where $\xi_{\ell}$ is a primitive $\ell$th root of unity):
\begin{equation}\label{eq1}
\tau(\psi) := -\sm_{x \in \F_\ell^\times} \psi(x)\,\xi_{\ell}^{x}
\in \Z[\mu_{p\,\ell}^{}].
\end{equation}
 
\begin{lemma}\label{congr}
We have $\tau(\psi)^{\sigma_a}
= \psi(a)^{-a}\, \tau(\psi^a)$, where $\sigma_a$
is the Artin automorphism attached to $a$
in ${\rm Gal}(\Q(\mu_{p\,\ell})/\Q)$, and
$\tau(\psi)^p \in \Z [\zeta_p]$; then
$\tau(\psi) \equiv 1 \pmod {{\mathfrak p}\,\Z[\mu_{p\,\ell}]}$.
\end{lemma}

\begin{proof}
By definition of $\sigma_a$, one has
$\tau(\psi)^{\sigma_a} = 
- \sm_{x \in \F_\ell^\times} \psi(x)^a\,\xi_{\ell}^{a\,x}
= - \psi^a(a^{-1})  \sm_{y \in \F_\ell^\times}\psi^a(y)\,\xi_{\ell}^{y}$;
whence the second claim taking 
$\sigma_a \in {\rm Gal}(\Q(\mu_{p\,\ell})/K)$ (i.e., $a \equiv 1 \pmod p$).

\smallskip
Then $\tau(\psi) \equiv -\sum_{x \in \F_\ell^\times} \xi_\ell^x \!
\pmod {{\mathfrak p}\,\Z[\mu_{p\,\ell}]}$; 
since $\ell$ is prime, $\sum_{x \in \F_\ell^\times} \xi_\ell^x = -1$.
\end{proof}

We then have the fundamental classical relation in $K$ 
(see \cite[\S\S\,6.1, 6.2, 15.1]{Wa}):
\begin{equation}\label{eq2}
{\mathfrak L}^{\,p S} = \tau(\psi)^p \, \Z[\zeta_p], 
\end{equation}

for ${\mathfrak L} \mid \ell$ such that $\psi$ is defined on the 
multiplicative group of $\Z[\zeta_p]/{\mathfrak L} \simeq \F_\ell$.

\begin{remark}\label{remaell}
(i) Since various choices of ${\mathfrak L} \mid \ell$, $\xi_\ell$ and $\psi$, 
from a given $\ell$, corres\-pond to Galois conjugations and/or products by
a $p$th root of unity, we denote simply $\tau(\psi)$ such a Gauss sum, 
where $\psi$ is for instance the canonical character of order $p$; 
for convenience, we shall have in mind that $\ell$ {\it defines} 
such a $\tau(\psi)$ (and some other objects) in an 
obvious way. One verifies that the forthcoming properties 
($p$-primarities, Kummer radicals\,$\ldots$) do not depend 
on these choices especially because of the action of the 
$\theta$-components.

\smallskip
(ii) If we consider $\alpha := \tau(\psi)^p \in K^\times$ as the 
Kummer radical of the cyclic extension $M_\ell := K(\tau(\psi))$ 
of $K$, we have $\alpha^{c - s_c} =: {\rm g}_c(\ell)^p$, where
${\rm g}_c(\ell) := \tau(\psi)^{c-\sigma_c} \in K^\times$; which 
gives $M_\ell = K(\sqrt[p]{\alpha}) = F_\ell K$, where $F_\ell$
is the subfield of $\Q(\mu_\ell)$ of degree $p$ (the character 
of $\langle \alpha \rangle K^{\times p}/K^{\times p}$
is $\omega$ and that of ${\rm Gal}(M_\ell/K)$ is~$1$). 
Thus $p$ is unramified in $M_\ell/K$
(which is coherent with $\tau(\psi) \equiv 1 \pmod 
{{\mathfrak p}\Z_p[\mu_{p\, \ell}]}$ implying 
$\tau(\psi)^p \equiv 1 \pmod {{\mathfrak p}^p}$); it splits if and 
only if $\tau(\psi)^p \equiv 1 \pmod {{\mathfrak p}^{p+1}}$.
\end{remark}

Taking the logarithms in \eqref{eq2}, we obtain, for all 
$\chi \in {\mathscr X}_+$:
$$\big (S\,.\, e_{\chi*} \big) \,.\, {\rm log} ({\mathfrak L}) =
 B_{1, \,(\chi^*)^{-1}} \,.\,{\rm log}({\mathfrak L}) \,.\, e_{\chi*}=
{\rm log}(\tau(\psi)) \,.\, e_{\chi*} ,$$

where ${\rm log}(\tau(\psi)):= \frac{1}{p}{\rm log}(\tau(\psi)^p) 
\in \Z_p[\varpi]$. Put $B_{1, \,(\chi^*)^{-1}} \sim p^e$, $e \geq 1$,
where $\sim$ means equality up to a $p$-adic unit.
Then  $p^e\,\Z_p \, {\rm log}({\mathfrak L}) 
\,.\, e_{\chi*} = {\Z_p}\, {\rm log}(\tau(\psi)) \,.\, e_{\chi*}$, 
thus, from \eqref{tauchi}, since $I/P$ may be represented by
prime ideals of degree $1$:
\begin{equation}\label{tchi*}
\order {\mathcal T}_{\chi*} = \frac {p^e}
{\order \big({\Z_p}{\rm log}\,({\mathcal G}) \big / 
p^e \, {\rm log}\,(U)\big)_{\chi*}}\, ,
\end{equation}

where ${\mathcal G}$ is the group generated by all the 
previous Gauss sums.

\smallskip
So, the ``Vandiver conjecture at $\chi \in {\mathscr X}_+$'' is equivalent to
$\big({\Z_p}\,{\rm log}\,({\mathcal G}) /{\rm log}(U)\big)_{\chi*} = 1$, 
and is, as expected, obviously fulfilled if $e=0$.
The whole Vandiver conjecture is equivalent to 
the fact that the images of the Gauss sums in
$U$ generate the minus part of this $\Z_p$-module
giving again Iwasawa's result \cite{Iw}. 

\smallskip
We shall from now make the following working hypothesis which 
corresponds to the more subtle case for testing Vandiver's conjecture 
with Theorems \ref{thmp}, \ref{N} (or Theorem \ref{first}), the case where
some $\Cl_{\chi*}$ are not cyclic being obvious for 
all the forthcoming statements, as soon as one knows that 
$B_{1, \,(\chi^*)^{-1}} \sim p^e$ gives the order of $\Cl_{\chi*}$ thus its 
annihilation and identities of the form ${\mathfrak a}^{p^e} = (\beta^p)$, 
$\beta \in K^\times$.
So, this will give ${\mathscr E}_\ell(p) \cap {\mathscr E}_0(p) \ne \ev$
for all $\ell\equiv 1\! \pmod p$ (see \S\,\ref{mainthm}):

\begin{hypothesis}[Cyclicity hypothesis]\label{hypo}
We assume that,  for all $\chi \in {\mathscr X}_+$, the component 
$\Cl_{\chi*}$ of the $p$-class group is cyclic (which implies the
cyclicity of $\Cl_{\chi}$); in other words, we restrict 
ourselves to the case where $\Cl$ is $\Z_p[G]$-monogenous 
(cf. Definition \ref{ND}\,(viii)), giving ${\rm rk}_p(\Cl_-) = i(p)$.
\end{hypothesis}

\subsection{Vandiver's conjecture and ray class group modulo $(p)$}
\label{appli}
Assume the Hypothesis \ref{hypo} and let $\chi = \omega^n 
\in {\mathscr X}_+$ be such that $B_{1, \,(\chi^*)^{-1}} \sim p^e$, $e \geq 1$
(i.e., $\Cl_{\chi*} \simeq \Z/p^e\Z$); thus, from \eqref{tchi*}, 
we have ${\mathcal T}_{\chi*}=1$ (i.e., $\Cl_\chi = 1$) if and only if there 
exists a prime number $\ell \equiv 1 \pmod p$ such that the corresponding 
${\rm log}(\tau(\psi)_{\chi*})$ generates 
${\rm log}(U_{\chi*}) = {\rm log} (1+\varpi^{p-n}\Z_p[\varpi])
= \varpi^{p-n}\Z_p[\varpi]$ (Proposition \ref{varpi}), which 
indicates analytically the non-$p$-primarity of 
$\tau(\psi)_{\chi*}$ in $\Z[\zeta_p]$ since $n>1$. 

\smallskip
There is also the fact that the Gauss sums (or the ${\rm g}_c(\ell)$), 
considered modulo $p$th powers and computed modulo $p$, are indexed
by infinitely many $\ell$; in other words there are some non-obvious large 
periodicities in the results as $\ell$ varies since numerical data
are finite in number.

\smallskip
This may be explained as follows (giving also an interesting criterion
which will imply new heuristics):
 
\begin{theorem}\label{cyclicity} 
Let $\Cl^{(p)}$ be the $p$-subgroup of  the ray 
class group $I/\{(x),\,  x \equiv 1 \pmod p\}$ of modulus $p\,\Z[\zeta_p]$. 
Then for any $\chi \in {\mathscr X}_+$,
we have (under the Hypothesis \ref{hypo}) the following properties:

\smallskip
(i) $\order \Cl^{(p)}_{\chi*} = p \cdot \order \Cl_{\chi*}$.

\smallskip
(ii) The condition $\Cl_\chi = 1$ is equivalent to the cyclicity of $\Cl^{(p)}_{\chi*}$.
\end{theorem}

\begin{proof} 
Let  $V := \{x \in K^\times \!, \, x \equiv 1\!\! \pmod {\mathfrak p}\}$ and
$W := \{x \in K^\times \!, \, x \equiv 1  \!\! \pmod p\}$.
Since $E_{\chi*}=1$, we have the exact sequence (using Proposition \ref{varpi}): 
$$1 \to (V/W)_{\chi*} \simeq \F_p \too \Cl^{(p)}_{\chi*} \too \Cl_{\chi*} \to 1,$$
giving (i). The statement (ii) is obvious if $\Cl_{\chi*}=1$. 
Suppose $\order \Cl_{\chi*}=p^e$, with  $e\geq 1$.

\smallskip
Then $\Cl_\chi = 1$ implies ${\mathcal T}_{\chi*}=1$ 
(from Theorem \ref{spiegel}) which implies
$\Cl^{(p)}_{\chi*} \simeq \Z/p^{e+1} \Z$: indeed, the $\chi^*$-part
$H^{\rm pr}_{\chi*}/K$ of the pro-$p$-extension $H^{\rm pr}/K$ is a
$\Z_p$-extension, thus the $p$-ray class field 
corresponding to $\Cl^{(p)}_{\chi*}$, contained in 
$H^{\rm pr}_{\chi*}$, is a cyclic extension of $K$.

\smallskip
Reciprocally, if $\Cl^{(p)}_{\chi*}  \simeq \Z/p^{e+1} \Z$, $e\geq 1$
(thus $\Cl_{\chi*} \simeq \Z/p^e \Z$), there exists 
${\mathfrak A}$ (whose class generates $\Cl^{(p)}_{\chi*}$)
such that ${\mathfrak A}_{\chi*}^{p^e} = (\alpha_{\chi*})$ 
(where $\alpha_{\chi*}$ is unique up to a $p$th power since $E_{\chi*}=1$)
with $\alpha_{\chi*} \equiv 1 \pmod {{\mathfrak p}^{p-n}}$
($\chi =: \omega^n$, $n\in [2, p-3]$ even), but 
$\alpha_{\chi*} \not\equiv 1 \pmod p$. Note that 
${\rm rk}_p({\mathcal T}_{\chi}) = {\rm rk}_p(\Cl_{\chi*}) = 1$.
Thus $\alpha_{\chi*}$ defines the radical of the unique $p$-ramified
(but not unramified) cyclic extension of degree $p$ of $K$
decomposed over $K_+$ into $L_+/K_+$
and contained in $H^{\rm pr}_{\chi}$ 
(its Galois group is a quotient of order $p$ of the {\it cyclic 
group} ${\mathcal T}_\chi$ since $\Gamma_\chi=1$
for an even $\chi \ne 1$); thus $\Cl_\chi = 1$.
\end{proof}

\section{Twists of Gauss sums associated to 
primes $\ell \equiv 1 \pmod p$}\label{sec4}

Let ${\mathscr L}_p$ be the set of primes $\ell$ totally split in $K$ 
(namely, $\ell \equiv 1 \pmod p$). For $\ell \in {\mathscr L}_p$, let
$\psi :  \F_\ell^\times \to \mu_p^{}$ be a multiplive character of order $p$; 
if $g$ is a primitive root modulo $\ell$, 
we put $\psi(g\! \pmod \ell) = \zeta_p$. 
Let $\xi_\ell$ be a primitive $\ell$-th root of unity; then the Gauss sum 
associated to $\ell$ may be written in $\Z[\mu_{p\,\ell}]$:
\begin{equation}\label{defk}
\tau(\psi) := -\sm_{x \in \F_\ell^\times} \psi(x) \cdot \xi_\ell^x =
-\sm_{k=0}^{\ell-2} \zeta_p^k \cdot \xi_\ell^{g^k}. 
\end{equation}

\subsection{Computation and properties of the twists ${\rm g}_c(\ell) := 
\tau(\psi)^{c - \sigma_c}$}
Let $c \in [2, p-2]$ be a primitive root modulo $p$; to get an integer 
of $K$ (a PARI/GP program in $\Z[\mu_{p\,\ell}]$ overflows as $\ell$ 
increases, even if $\tau(\psi)_{\chi*} = \tau(\psi)^{e'_{\chi*}}$ 
makes sense in $\Z[\zeta_p]$, a posteriori), one uses the twist 
$\tau(\psi)^{c - \sigma_c}$, where $\sigma_c$
is the Artin automorphism attached to $c$
in ${\rm Gal}(\Q(\mu_{p\,\ell}^{})/\Q)$. 
We define for $\ell \in {\mathscr L}_p$ (cf. Lemma \ref{congr}):
\begin{equation}\label{tcl}
{\rm g}_c(\ell) := \tau(\psi)^{c - \sigma_c} \in \Z[\zeta_p] \  \,
\hbox{(see formulas \eqref{eq1}, \eqref{eq2} and Remark \ref{remaell})}.
\end{equation}

giving for all $\chi \in {\mathscr X}_+$, up to $K^{\times p}$ for the generators of ideals:
$${\mathfrak L}^{S_c}\! = {\rm g}_c(\ell) \, \Z[\zeta_p]  
\ \ \& \ \ {\mathfrak L}_{\chi*}^{(c -  \chi*(s_c)) \cdot B_{1, \,(\chi^*)^{-1}}}\! = 
{\rm g}_c(\ell)_{\chi*} \, \Z[\zeta_p]$$ 

(see Definitions \ref{ND}),
where ${\mathfrak L} \mid \ell$ in $K$,
$S_c := (c - s_c) \cdot S \in \Z[G]$ is the corresponding 
twist of the Stickelberger element and where 
${\rm g}_c(\ell) \in \Z[\zeta_p]$. Put:
\begin{equation}\label{eq3}
b_c(\chi^*) := (c -  \chi^*(s_c)) \cdot B_{1, \,(\chi^*)^{-1}} \sim B_{1, \,(\chi^*)^{-1}}, 
\hbox{ for all $\chi \in {\mathscr X}_+$}. 
\end{equation}

Then we obtain the main relation that will be of a constant use:
\begin{equation}\label{eqfond}
 {\mathfrak L}_{\chi*}^{b_c(\chi*)}= {\rm g}_c(\ell)_{\chi*}\,\Z[\zeta_p] .
\end{equation}

\begin{remark}\label{gooddef}
(i) In the above definition \eqref{tcl} of ${\rm g}_c(\ell)$,
$\tau(\psi)^{\sigma_c} = \tau(\psi^c)\cdot \psi^{-c}(c)$
(Lemma \ref{congr}); but for all $\chi \ne 1$, 
$\mu_p^{e_{\chi*}}=1$, defining ${\rm g}_c(\ell)_{\chi*}$ 
without ambiguity up to $K^{\times p}$, which does
not change the $p$-primarity properties. 
But in some sense the best definition of the twists should be
$\psi^{-c}(c)\cdot {\rm g}_c(\ell) = \psi^{-c}(c)\cdot 
\tau(\psi)^{c-\sigma_c}$.

\smallskip
(ii) Note that, since $\tau(\psi)^{1+s_{-1}} = \ell$, this yields
${\rm g}_c(\ell)_\chi \in K^{\times p}$ for all $\chi \in {\mathscr X}_+$.
\end{remark}

\begin{lemma} \label{tc}
Let $\ell \in {\mathscr L}_p$ be given. Then 
$\psi^{-c}(c)\cdot {\rm g}_c(\ell)$ is a product of Jacobi sums 
and $\psi^{-c}(c)\cdot {\rm g}_c(\ell) \equiv
{\rm g}_c(\ell) \equiv 1 \pmod {\mathfrak p}$.
\end{lemma}

\begin{proof} The classical formula \cite[\S\,6.1]{Wa}
for Jacobi sums (with $\psi\,\psi' \ne 1$) is:
$$J(\psi, \psi') :=  \tau(\psi) \cdot \tau(\psi')\cdot \tau(\psi\,\psi')^{-1} =
- \sm_{x \in \F_\ell \, \setminus  \, \{0,1\}} \psi(x) \cdot \psi' (1-x). $$ 

Whence
$\tau(\psi)^c = J_1\cdots J_{c-1} \cdot \tau(\psi^c)$, 
where $J_i = -  \!\!\sm_{x \in \F_\ell \, \setminus  \, \{0,1\}} 
\psi^i(x) \cdot \psi (1-x)$, thus: 
\begin{equation*}
\tau(\psi)^{c-\sigma_c} = J_1\cdots J_{c-1} \cdot \tau(\psi^c)\,
\tau(\psi)^{-\sigma_c} = J_1\cdots J_{c-1} \cdot \psi^{c}(c),
\end{equation*}

from Lemma \ref{congr}; then
$\tau(\psi) \equiv 1 \!\!\pmod {{\mathfrak p}\,\Z[\mu_{p\,\ell}]}$
implies the result for ${\rm g}_c(\ell)$.
\end{proof}

Thus, in the numerical computations, we shall use the relation:
\begin{equation}\label{jacobi}
{\rm g}_c(\ell)_{\chi*} = ( J_1\cdots J_{c-1})_{\chi*}\ 
\hbox{for any $\chi \in {\mathscr X}_+$}.
\end{equation}

The following definitions will be of constant use in the paper:

\begin{definition}[exponents of $p$-primarity 
and $p$-irregularity] \label{defsets}
${}$
\quad (i) We call set of exponents of $p$-primarity, of a prime 
$\ell \in {\mathscr L}_p$, the set 
${\mathscr E}_\ell(p)$ of even integers $n \in [2, p-3]$  
such that ${\rm g}_c(\ell)_{\omega^{p-n}}$ is $p$-primary,
thus ${\rm g}_c(\ell)_{\omega^{p-n}} \equiv 1\!\! \pmod p$
(Definition \ref{psu}\,(ii), Proposition \ref{varpi}).

\smallskip
\quad (ii) We call set of exponents of $p$-irregularity, the set 
${\mathscr E}_0(p)$ of even integers $n \in [2, p-3]$  
such that $B_n \equiv 0 \!\pmod p$,
thus, $B_{1,{\omega^{n-1}}} \equiv 0\! \pmod p$ (see Definitions \ref{ND}\,(vii)).
\end{definition}

\begin{remark} \label{remCNS}
Let $\chi =: \omega^n \in {\mathscr X}_+$ and $\ell \in {\mathscr L}_p$.
If ${\rm g}_c(\ell)_{\chi*}$ is 
$p$-primary ($n \in {\mathscr E}_\ell(p)$) this does not give 
necessarily a counterexample to Vandiver's conjecture for the
two following possible reasons considering  
$S_c \, e_{\chi*} = b_c({\chi^*}) \, e_{\chi*}$; recall that from \eqref{eq3}, 
$$b_c({\chi^*}) = 
(c - \chi^*(s_c) )\cdot B_{1, \,(\chi^*)^{-1}} \sim 
B_{1, \,(\chi^*)^{-1}} = B_{1, \, \omega^{n-1}}.$$

\quad (i) The number $b_c({\chi^*})$ is a $p$-adic unit ($n \notin {\mathscr E}_0(p)$),
so the radical ${\rm g}_c(\ell)_{\chi*}$ is not the $p$th power of an ideal (thus not a
pseudo-unit, even if Proposition \ref{varpi} applies) and leads to a cyclic
$\ell$-ramified Kummer extension of degree $p$ of $K_+$.

\smallskip
For instance, for $p=11$ ($c=2$), $\ell = 23$,
the exponent of $11$-primarity is $n=2$ so that $\alpha := {\rm g}_c(\ell)_{\chi*}$
is the integer (where $x=\zeta_{11}$):

\footnotesize
\smallskip
\begin{verbatim}
-8491773970656065727678427465045288222*x^9-1963231019856677733688722439078492228*x^8
+11757523232198873159205810348854526320*x^7-5860674150310922200348907606983566648*x^6
-644088006192816851608142123579276962*x^5-611074014289231284308386817199658010*x^4
+2673005955545675004066087284224877298*x^3+15023028737838809151251842166615658188*x^2
+1520229819300797188419125563036321734*x+17836238554732163868933693789025679469
\end{verbatim}

\normalsize
\smallskip
for which $K(\sqrt[11]{\alpha})/K$ is decomposed over $K_+$ into 
$L_+/K_+$, $\ell$-ramified; then $(\alpha)$ is a product of 
prime ideals above $\ell$ ($s = s_2$): $(\alpha) = 
{\mathfrak L}^{1+2s+2^2s^2+2^3s^3+2^4s^4+2^5s^5+
2^6s^6+2^7s^7+2^8s^8+2^9s^9}$,
up to the $11$th power of an $\ell$-ideal.
We get ${\rm N}_{K/\Q}(\alpha)=\ell^{275}$ and
${\rm N}_{K/\Q}(\alpha-1) \sim 11^{13}$. In fact the program gives
$(\alpha) = {\mathfrak L}_1^{25} \!\cdot\! {\mathfrak L}_2^{27} 
\!\cdot\! {\mathfrak L}_3^{31}
\!\cdot\! {\mathfrak L}_4^{24} \!\cdot\! {\mathfrak L}_5^{28} 
\!\cdot\! {\mathfrak L}_6^{15}
\!\cdot\! {\mathfrak L}_7^{30} \!\cdot\! {\mathfrak L}_8^{23} 
\!\cdot\! {\mathfrak L}_9^{32}
\!\cdot\!{\mathfrak L}_{10}^{40}$ 
and one must discover the significance given above~!
Here $b_c({\chi^*}) \equiv 1 \pmod {11}$.

\smallskip
\quad (ii) The number $b_c({\chi^*})$ is divisible by $p$, 
but the ideal ${\mathfrak L}_{\chi*}$ is $p$-principal and then 
${\rm g}_c(\ell)_{\chi*}$ is a $p$th power in $K^\times$
(numerical examples in \S\,\ref{classes37}).
\end{remark}

\subsection{First main theorem}\label{mainthm}

So, from the previous Remark \ref{remCNS}, a {\it sufficient condition 
for the existence of a counterexample} to Vandiver's conjecture is the 
existence of $\chi \in {\mathscr X}_+$ and $\ell \in {\mathscr L}_p$
such that the three following conditions are fulfilled:

\smallskip
(a) $b_c(\chi^*) \equiv 0 \pmod p$, 

(b) ${\rm g}_c(\ell)_{\chi*}$ is $p$-primary,

(c) ${\rm g}_c(\ell)_{\chi*}$ is not a global $p$th power.

\medskip
We make here a fundamental remark:

\begin{remark} \label{remfond}
If ${\rm rk}_p(\Cl_{\chi_0^*}) \geq 2$ for $\chi_0^{} = \omega^{n_0} \in {\mathscr X}_+$ 
(giving a counterexample to Vandiver's conjecture), we get, from the ``Main Theorem'',
$\order \Cl_{\chi_0^*} \sim b_c({\chi_0^*})$; then the $p$-part of $b_c({\chi_0^*})$ is 
strictely larger than the exponent of $\Cl_{\chi_0^*}$ so that, in any relation 
${\mathfrak L}_{\chi_0^*}^{b_c({\chi_0^*})} = ({\rm g}_c(\ell)_{\chi_0^*})$ 
where ${\mathfrak L}_{\chi_0^*}$ define a generating class of $\Cl_{\chi_0^*}$,
necessarily ${\rm g}_c(\ell)_{\chi_0^*}$ is 
a global $p$th power (condition (c) is never fulfilled), whence the property
$n_0 \in {\mathscr E}_\ell(p) \cap {\mathscr E}_0(p) \ne \ev$ for all $\ell \in {\mathscr L}_p$; 
thus Theorems \ref{thmp} and \ref{N} will apply for trivial reasons and we can 
go back to the cases ${\rm rk}_p(\Cl_{\chi*}) < 2$ (Hypothesis \ref{hypo}) for the reciprocal.
\end{remark}

\begin{lemma} \label{lem}
Let $\chi \in {\mathscr X}_+$ such that $\Cl_\chi \ne 1$. 
There exists a totally split prime ideal ${\mathfrak L}$
such that ${\mathfrak L}_{\chi*}$ represents a generator of
$\Cl_{\chi*}$.
Then ${\mathfrak L}^{S_c \, e_{\chi*}} 
= {\mathfrak L}_{\chi*}^{b_c(\chi*)}  = (\alpha_{\chi*})$, 
where $\alpha_{\chi*}$ is unique 
(up to a $p$th power), thus equal to ${\rm g}_c(\ell)_{\chi*}$
which is $p$-primary and not a global $p$th power.
\end{lemma}

\begin{proof} 
From the Chebotarev density theorem in $H/\Q$, 
there exists a prime $\ell$ and $\ov {\mathfrak L} \mid \ell$ 
in $H$ such that the Frobenius
$\big( \frac{H/\Q}{ \ov {\mathfrak L}} \big)$ generates the subgroup of
${\rm Gal}(H/K)$ corresponding to $\Cl_{\chi*}$ by class field theory.
So $\ell$ splits completely in $K/\Q$ ($\ell \in {\mathscr L}_p$)
and the ideal ${\mathfrak L}$ of $K$ under $\ov {\mathfrak L}$ is 
(as ${\mathfrak L}_{\chi*}$) a representative of 
a {\it generator} of $\Cl_{\chi*} \simeq \Z_p/b_c(\chi^*)\,\Z_p$. 
Then ${\mathfrak L}_{\chi*}^{b_c(\chi^*)} = (\alpha_{\chi*})$
where $\alpha_{\chi*} \notin K^{\times p}$; $\alpha_{\chi*}$ is unique since 
$E_{\chi*} = 1$ for $\chi^*\ne \omega$. In terms 
of Gauss sums, ${\mathfrak L}_{\chi*}^{b_c(\chi^*)} = 
({\rm g}_c(\ell)_{\chi*})$, thus $\alpha_{\chi*} = {\rm g}_c(\ell)_{\chi*}$. 
The $p$-primarity of $\alpha_{\chi*}$ is necessary to obtain 
the {\it unique} (still thanks to Hypothesis \ref{hypo}) unramified 
Kummer extension $K (\sqrt[p]{\alpha_{\chi*}})/K$ of degree $p$,
decomposed over $K_+$ into the unramified extension
$L_+/K_+$ of degree $p$ in $H_\chi$, associated to $\Cl_\chi/\Cl_\chi^p$ by 
class field theory, whence the $p$-primarity of ${\rm g}_c(\ell)_{\chi*}$.
\end{proof}

Drawing the consequences of the above, we get,
unconditionally, the main test for Vandiver's conjecture
stated in the Introduction (Theorem \ref{first}\,(a)).
We refer to the relations \eqref{tcl}, \eqref{eq3}, \eqref{eqfond}
and the Definition \ref{defsets}.

\smallskip
\begin{theorem} \label{thmp}
Vandiver's conjecture holds for $K = \Q(\mu_p)$ if and only if
there exists $\ell \equiv 1 \pmod p$ such that
${\mathscr E}_\ell(p) \cap {\mathscr E}_0(p)=\ev$.
\end{theorem}

\begin{proof} As explained in the Remark \ref{remfond}, 
we may assume the cyclicity Hypothesis \ref{hypo}.

\smallskip
Suppose ${\mathscr E}_\ell(p) \cap {\mathscr E}_0(p)=\ev$ and consider, 
for $\chi =: \omega^n \in {\mathscr X}_+$, and $\chi^*=\omega^{p-n}$, the 
relation ${\mathfrak L}_{\chi*}^{b_c(\chi*)} = ({\rm g}_c(\ell)_{\chi*})$ for the 
prime $\ell$ under consideration, and examine the two possibilities:

\smallskip
\quad (i) If $n$ is not an exponent of $p$-irregularity (namely, 
$b_c({\chi^*}) \not \equiv 0 \pmod p$ or $B_n \not\equiv 0 \pmod p$), then
$\Cl_{\chi*}=1$ and $\Cl_{\chi}=1$ from reflection theorem 
(Corollary \ref{reflection}).

\smallskip
\quad (ii) If $n$ is an exponent of $p$-irregularity, then 
$b_c({\chi^*}) \sim p^e$, $e \geq 1$, giving, for some 
$p$-adic unit $u$, ${\mathfrak L}_{\chi*}^{p^e u} = 
({\rm g}_c(\ell)_{\chi*})$ (Lemma \ref{lem}); if ${\mathfrak L}_{\chi*}^{p^{e-1}u}$
is $p$-principal, then ${\rm g}_c(\ell)_{\chi*}$
is a global $p$th power, hence $p$-primary (absurd by assumption).
So ${\mathfrak L}_{\chi*}$ defines a class of order $p^e$ in $\Cl_{\chi*}$
for which the pseudo-unit ${\rm g}_c(\ell)_{\chi*}$ is not $p$-primary by 
assumption; since ${\rm Gal}(H^{\rm pr}_\chi/K_+) = 
{\mathcal T}_\chi$ is cyclic, from relation \eqref{spiegel},
 by Kummer duality $K(\sqrt[p]{{\rm g}_c(\ell)_{\chi*}})$ is the unique 
extension cyclic of degree $p$, decomposed over $K_+$ and
contained in $H^{\rm pr}_\chi$. Since it is ramified at $p$ and since
$H^{\rm pr}_\chi$ contains the $\chi$-component of the $p$-Hilbert 
class field of $K_+$, this implies $\Cl_{\chi}=1$.

\smallskip
Reciprocally, if Vandiver's conjecture holds, then
$\Cl=\Cl_-$ is $\Z_p[G]$-monogenous, thus the direct sum
of non-trivial cyclic isotypic components generated by some $p$-classes
$\gamma^{(n_i)} = \cl({\mathfrak L}^{(n_i)}_{\omega^{p-n_i}}) 
\in \Cl_{\omega^{p-n_i}}$
($n_i \in {\mathscr E}_0(p)$) related to non-$p$-primary 
${\rm g}_c(\ell^{(n_i)})_{\omega^{p-n_i}}$;
thus there exists, from density theorem, $\ell \in {\mathscr L}_p$
and ${\mathfrak L} \mid \ell$ such that 
$\cl({\mathfrak L})_{\omega^{p-n_i}}=
\gamma^{(n_i)}$ for all $i$ (e.g.,
${\mathfrak L} = (z) \cdot \prod_{i} 
{\mathfrak L}^{(n_i)}_{\omega^{p-n_i}}$). So each 
${\rm g}_c(\ell)_{\omega^{p-n_i}} = 
{\rm g}_c(\ell^{(n_i)})_{\omega^{p-n_i}}$ (up to a $p$th power)
is non-$p$-primary, whence 
${\mathscr E}_\ell(p) \cap {\mathscr E}_0(p)=\ev$ for this prime $\ell$.
\end{proof}

\begin{corollary} \label{casvide}
Let $\ell \in {\mathscr L}_p$. If, for all $\chi \in {\mathscr X}_+$,
the numbers ${\rm g}_c(\ell)_{\chi*}$ 
are not $p$-primary (i.e., ${\mathscr E}_\ell(p) = \ev$),
then the Vandiver conjecture holds for $p$.
\end{corollary}

\subsubsection{Program computing ${\mathscr E}_\ell(p)$.}
For $p\in [3, 199]$ and for the least $\ell \in {\mathscr L}_p$,
the program computes ${\rm g}_c(\ell)$ in 
${\sf Mod(J,P)}$, with ${\sf P=polcyclo(p)}$, where 
${\sf J = J_1\cdots J_{c-1}}$ is written in $\Z[x]$ modulo $p\,\Z[x]$; 
${\sf c}$ is a primitive root (mod ${\sf p}$) (see the relation \eqref{jacobi}). 

\smallskip
For the computation of ${\sf J_i}$ we use the discrete logarithm
${\sf znlog}$ to interprete the $1-g^k$ in $g^{\Z/(\ell-1)\Z}$.
We put  $\chi=\omega^n \ \& \ \chi^*= \omega^{1-n}$,
taking ${\sf n=2*m}$ for ${\sf m}$ in $[1, (p-3)/2]$.

\smallskip
The program takes into account the relation 
$J^{1+s_{-1}}\equiv 1 \pmod p$ in the action of the idempotents
and drops the coefficient $\frac{1}{p-1}$ in $e_{\chi*}$ (in which
$\chi^*(s_a^{-1})$ is replaced by the residue of $a^{n-1}$ modulo $p$), 
thus computes in reality ${\rm g}_c(\ell)^{-1/2}$ up to $p$th powers.
Then the polynomials ${\sf Jj}$ give, in the list ${\sf LJ}$, the powers 
${\sf J^j}$ modulo $p$, $j=1,\ldots,p-1$. 

\smallskip
The result is given in 
$${\sf Sn = \prd_{a=1}^{(p-1)/2} s_a(J^{an})}, \ \ \ \hbox{from} \ \ \
{\rm g}_c(\ell)_{\chi*}^{-1/2} = 
\prd_{a=1}^{(p-1)/2} s_a \big ({\rm g}_c(\ell)^{\omega^{n-1}(a)}\big)$$
(up to a $p$th power), where $\omega^{n-1}(a) \equiv a^{n-1} \pmod p$
is computed in ${\sf an}$ and ${\sf J^{an}}$ is given by
${\sf component(LJ,an)}$. 
The conjugate ${\sf s_a(J^{an})}$ 
is computed in ${\sf sJan}$ via the conjugation 
${\sf x \mapsto x^a}$ in ${\sf J^{an}}$, whence the product in 
${\sf Sn}$ (the exponents of $p$-primarity are denoted ${\sf expp}$):

\medskip
{\bf Note:}\label{note} To copy and past the programs in verbatim text, 
one must perhaps replace the symbol of power (in a\^{}b) by the 
PARI/GP symbol ($=$ that of the keyboard); otherwise the 
program does not work (this is due to the character font used by some 
Journals).

\footnotesize
\smallskip
\begin{verbatim}
{forprime(p=3,200,c=lift(znprimroot(p));P=polcyclo(p)+Mod(0,p);
X=Mod(x,P);el=1;while(isprime(el)==0,el=el+2*p);g=znprimroot(el);
print("p=",p," el=",el," c=",c," g=",g);J=1;for(i=1,c-1,Ji=0;
for(k=1,el-2,kk=znlog(1-g^k,g);e=lift(Mod(kk+i*k,p));Ji=Ji-X^e);J=J*Ji);
LJ=List;Jj=1;for(j=1,p-1,Jj=lift(Jj*J);listinsert(LJ,Jj,j));
for(m=1,(p-3)/2,n=2*m;Sn=Mod(1,P);for(a=1,(p-1)/2,
an=lift(Mod(a,p)^(n-1));Jan=component(LJ,an);sJan=Mod(0,P);
for(j=0,p-2,aj=lift(Mod(a*j,p));sJan=sJan+x^(aj)*component(Jan,1+j));
Sn=Sn*sJan);if(Sn==1,print("   exponents of p-primarity: ",n))))}

p=3   el=7   c=2  g=3                p=97  el=389  c=5  g=2 expp:26
p=5   el=11  c=2  g=2                p=101 el=607  c=2  g=3 expp:10
p=7   el=29  c=2  g=2                p=103 el=619  c=5  g=3
p=11  el=23  c=3  g=5 expp:2         p=107 el=643  c=2  g=11
p=13  el=53  c=2  g=2                p=109 el=1091 c=6  g=2 expp:14,86
p=17  el=103 c=3  g=5                p=113 el=227  c=3  g=2
p=19  el=191 c=4  g=19               p=127 el=509  c=3  g=2
p=23  el=47  c=2  g=5                p=131 el=263  c=2  g=5 expp:16
p=29  el=59  c=2  g=2 expp:2         p=137 el=823  c=3  g=3 expp:78
p=31  el=311 c=7  g=17               p=139 el=557  c=2  g=2
p=37  el=149 c=2  g=2                p=149 el=1193 c=2  g=3
p=41  el=83  c=6  g=2                p=151 el=907  c=6  g=2
p=43  el=173 c=9  g=2 expp:26        p=157 el=1571 c=5  g=2 expp:94
p=47  el=283 c=2  g=3                p=163 el=653  c=2  g=2 expp:42
p=53  el=107 c=2  g=2 expp:34,10     p=167 el=2339 c=5  g=2 expp:122
p=59  el=709 c=3  g=2                p=173 el=347  c=2  g=2
p=61  el=367 c=2  g=6                p=179 el=359  c=2  g=7 expp:138
p=67  el=269 c=4  g=2                p=181 el=1087 c=2  g=3 expp:114,164
p=71  el=569 c=2  g=3                p=191 el=383  c=19 g=5
p=73  el=293 c=5  g=2                p=193 el=773  c=5  g=2 expp:108,172
p=79  el=317 c=2  g=2                p=197 el=3547 c=2  g=2 expp:62
p=83  el=167 c=3  g=5                p=199 el=797  c=3  g=2
p=89  el=179 c=3  g=2             
 \end{verbatim}
\normalsize

\subsubsection{Minimal prime $\ell \in {\mathscr L}_p$ 
such that ${\mathscr E}_\ell(p) = \ev$.} \label{exist}
The following program examines, for each $p$, the successive prime
numbers $\ell_i \in {\mathscr L}_p$, $i \geq 1$,
and returns the first one, $\ell_N$ (in ${\sf L}$ with its index ${\sf N}$), 
such that ${\mathscr E}_{\ell_N}(p)=\ev$.
Its existence is of course a strong conjecture, but the numerical 
results are extremely favorable to the existence of 
such primes; which strengthens the conjecture of Vandiver.
Moreover, since the integer $i(p) = \order {\mathscr E}_0(p)$ is 
rather small regarding $p$, as doubtless for $\order {\mathscr E}_\ell(p)$,
and can be both in $O\big(\frac{{\rm log}(p)}{{\rm log}({\rm log}(p))} \big)$,
the intersection ${\mathscr E}_\ell(p) \cap {\mathscr E}_0(p)$ 
may be easily empty {\it if these sets are independent}.
\smallskip
The experiments give the impression that the sets ${\mathscr E}_\ell(p)$ 
are random when $\ell$ varies and have no link with ${\mathscr E}_0(p)$.

\footnotesize
\smallskip
\begin{verbatim}
{forprime(p=3,100,c=lift(znprimroot(p));P=polcyclo(p)+Mod(0,p);
N=0;T=1;el=1;while(T==1,el=el+2*p;if(isprime(el)==1,N=N+1;g=znprimroot(el);
J=Mod(1,P);for(i=1,c-1,Ji=0;for(k=1,el-2,kk=znlog(1-g^k,g);
e=lift(Mod(kk+i*k,p));Ji=Ji-x^e);J=J*Ji);LJ=List;Jj=1;
for(j=1,p-1,Jj=lift(Jj*J);listinsert(LJ,Jj,j));T=0;for(m=1,(p-3)/2,n=2*m;
Sn=Mod(1,P);for(a=1,(p-1)/2,an=lift(Mod(a,p)^(n-1));Jan=component(LJ,an);
sJan=0;for(j=0,p-2,aj=lift(Mod(a*j,p));sJan=sJan+x^(aj)*component(Jan,1+j));
Sn=Sn*sJan);if(Sn==1,T=1;break));if(T==0,print(p," ",el," ",N);break))))}
\end{verbatim}

\normalsize
\smallskip
For $p<400$, we only write the primes $p, \ell_N$ for which $N>1$,
then a complete list for $p \in [409, 683]$:

\footnotesize
\smallskip
\begin{verbatim}
p   el    N    p   el    N          p   el    N    p   el    N    p   el     N
11  67    2    197 4729  2          409 4091  2    499 1997  1    601 25243  5
29  233   2    211 10973 4          419 839   1    503 3019  1    607 20639  3
43  431   2    223 6691  2          421 4211  1    509 4073  2    613 6131   1
53  743   2    227 5903  2          431 863   1    521 16673 1    617 30851  3
97  971   2    229 5039  2          433 5197  2    523 6277  2    619 17333  3
101 809   2    233 1399  2          439 4391  1    541 9739  1    631 6311   1
109 2399  2    251 4519  2          443 887   1    547 5471  1    641 1283   1
131 1049  3    277 4987  3          449 3593  1    557 24509 3    643 10289  2   
137 1097  2    337 6067  3          457 21023 3    563 7883  1    647 9059   1
157 7537  5    349 8377  2          461 9221  2    569 6829  1    653 1307   1
163 5869  3    367 3671  2          463 5557  1    571 5711  1    659 1319   1
167 7349  3    383 16087 4          467 2803  1    577 3463  2    661 14543  3
179 1433  2    389 14783 2          479 3833  1    587 8219  1    673 2693   1
181 1811  2    397 6353  2          487 1949  1    593 1187  1    677 5417   1
193 1931  2    401 10427 4          491 983   1    599 4793  1    683 4099   2
\end{verbatim}
\normalsize

\smallskip
The comparison with the table of exponents of $p$-irregularity
does not show any relation.

\subsection{Second main theorem}\label{classes}\label{ellell'}
Let $n_0$ be an exponent of $p$-irregularity;
put $\chi_0^{} = \omega^{n_0}$ and let $b_c(\chi_0^*) \sim p^e$, 
$e \geq 1$. If $\Cl_{\chi_0^*}$ is not cyclic, Remark \ref{remfond}
implies $n_0 \in \cap_{\ell \in {\mathscr L}_p} {\mathscr E}_\ell(p) \ne \ev$
and Theorem \ref{N} will hold.
Then we may assume $\Cl_{\chi_0^*} \simeq \Z/ p^e \Z$. 
We shall examine what happens when $\ell \in {\mathscr L}_p$ varies.

\smallskip
Let $\ell \in {\mathscr L}_p$ and let ${\mathfrak L}_{\chi_0^*}$
with ${\mathfrak L} \mid \ell$.
There are two cases as we have seen previously in the 
monogenous case:

\smallskip
\quad (i)  ${\mathfrak L}_{\chi_0^*}^{p^{e-1}}$ is $p$-principal.
Since $b_c(\chi_0^*) \sim p^e$, $e \geq 1$,
${\rm g}_c(\ell)_{\chi_0^*}$ is a global $p$th power in $K^\times$, 
whence ${\rm g}_c(\ell)_{\chi_0^*}$ is $p$-primary and 
$n_0 \in {\mathscr E}_\ell(p)$, but this does not lead
to an unramified cyclic extension of degree $p$ of $K_+$ 
of character~$\chi_0^{}$;

\smallskip
\quad (ii) ${\mathfrak L}_{\chi_0^*}^{p^{e-1}}$ is not $p$-principal
(from density theorem, such primes $\ell$ always exist).
Thus it defines a generator of $\Cl_{\chi_0^*}$ and Vandiver's
conjecture ``holds for $\chi_0^{} = \omega^{n_0}$'' if and only if 
${\rm g}_c(\ell)_{\chi_0^*}$ is not $p$-primary (Theorem \ref{thmp}).

\smallskip
If ${\rm g}_c(\ell)_{\chi_0^*} \equiv 1 \pmod p$
(counterexample to Vandiver's conjecture), 
we fix this $\ell$ once for all, and whatever the ideal 
${\mathfrak L}' \mid \ell' $, $\ell' \in {\mathscr L}_p$, we have 
${\mathfrak L}'_{\chi_0^*} = (z) \cdot {\mathfrak L}^r_{\chi_0^*}$,
with $z \in K^\times$ and $r \in [0, p^e-1]$, so:
$${\mathfrak L}'{}^{p^e u}_{\chi_0^*} = 
(z^{p^e u}) \cdot {\mathfrak L}^{r p^e u}_{\chi_0^*}\ \ \ \&
\ \ \  {\rm g}_c(\ell')_{\chi_0^*} \equiv 
{\rm g}_c(\ell)_{\chi_0^*}^r \equiv 1 \pmod p. $$

Whence, the exponent $n_0$ of $p$-irregularity is
a common exponent  of $p$-primarity for all 
$\ell \in {\mathscr L}_p$, giving $n_0 \in {\mathscr E}_0 (p) \cap 
\big(\cap_{\ell \in {\mathscr L}_p} \!{\mathscr E}_\ell(p) \big) \ne \ev$.
In other words, the existence of  an empty intersection
${\mathscr E}_{\ell_1}(p) \cap \cdots \cap {\mathscr E}_{\ell_N}(p)$
implies Vandiver's conjecture. We shall also prove the reciprocal,
that gives the new criterion:

\smallskip
\begin{theorem} \label{N}
Vandiver's conjecture holds if and only if there exist $N \geq 1$
and $\ell_1, \ldots , \ell_N \in {\mathscr L}_p$ such that 
${\mathscr E}_{\ell_1}(p) \cap \cdots \cap {\mathscr E}_{\ell_N}(p)=\ev$.
\end{theorem}

\begin{proof} 
It remains to prove that Vandiver's conjecture implies
such an empty intersection.
Assume, on the contrary, that for all $N \geq 1$ and all sets 
$\{\ell_1, \ldots , \ell_N\} \subset {\mathscr L}_p$, one has
${\mathscr E}_{\ell_1}(p) \cap \cdots \cap {\mathscr E}_{\ell_N}(p) \ne \ev$.

\smallskip
Since ${\mathscr X}_+$ is finite, there exists such an
$n_0$ in $\bigcap_{\ell \in {\mathscr L}_p} {\mathscr E}_{\ell}(p)$
(if $\cap_{\ell \in {\mathscr L}_p} {\mathscr E}_{\ell}(p) = \ev$ then for all even 
$n \in [2, p-3]$ there exists $\ell(n)$ such that $n \notin {\mathscr E}_{\ell(n)}(p)$
whence $\bigcap_{n \in [2, p-3]\, {\rm even}}\, {\mathscr E}_{\ell(n)}(p) = \ev$ (absurd)).
This means that for the fixed character $\chi_0^{} := \omega^{n_0}$, we
have the property:
$$\hbox{${\rm g}_c(\ell)_{\chi_0^*} \equiv 1\!\! \pmod p$, \ 
for all $\ell \in {\mathscr L}_p$.}$$

To simplify, put $\alpha(\ell) := {\rm g}_c(\ell)_{\chi_0^*}$ and 
consider the extensions $K(\sqrt[p]{\alpha(\ell)}) / K$; these extensions, 
with Galois groups of character $\chi_0^{}$, are decomposed over 
$K_+$ into cyclic extensions $L_+(\ell)/K_+$ (possibly trivial), and are
$\ell$-ramified since $(\alpha(\ell) ) = 
{\mathfrak L}_{\chi_0^*}^{b_c(\chi_0^*)}$ with 
$\alpha(\ell)\equiv 1 \pmod p$ (non-ramification at $p$). 
Examine the two possibilities about $b_c(\chi_0^*)$:

\smallskip
\quad (i) $b_c(\chi_0^*) \equiv 0 \pmod p$. Then $\alpha(\ell)$ is,
for all $\ell$, a $p$-primary pseudo-unit, and choosing $\ell$
such that ${\mathfrak L}_{\chi_0^*}$ generates $\Cl_{\chi_0^*}$
(which is cyclic since $\Cl_{\chi_0^{}}=1$), the extension 
$L_+(\ell)/K_+$ is unramified of degree $p$ (absurd).

\smallskip
\quad (ii) $b_c(\chi_0^*) \not\equiv 0 \pmod p$. Then $L_+(\ell)/K_+$ is, 
for all $\ell$, a {\it $\ell$-ramified degree $p$ cyclic extension
of character $\chi_0^{}$}. We restrict ourselves to primes 
$\ell \not \equiv 1 \pmod {p^2}$ and consi\-der the $p$-ray class
fields, $H_+(\ell)$ over $K_+$, of modulus $(\ell)$; we have
$L_+(\ell) \subseteq H_+(\ell)$. Since $\Cl_+ = 1$, 
${\rm Gal}(H_+(\ell)/K_+) \simeq (P_+/P_+(\ell)) \otimes \Z_p$, where
$P_+$ is the group of principal ideals prime to $\ell$ of $K_+$
and $P_+(\ell)$ the subgroup of $P_+$ of ideals generated by an
element $x \equiv 1 \pmod \ell$. 

\smallskip
From the $G$-modules exact sequence 
$1 \to E_+/E_+(\ell) \to\hbox{$ \plus_{{\mathfrak L}_+ \mid \ell}$ }
\F_\ell^\times \to  P_+/P_+(\ell) \to 1$,
where $E_+(\ell) := \{\varepsilon \in E_+, \  \varepsilon \equiv 1 \pmod \ell\}$, 
we get (for $\ell \not \equiv 1 \pmod {p^2}$):
$$1 \to (E_+/E_+(\ell))_{\chi_0^{}}  \too (\Z/p\Z)_{\chi_0^{}}  \too 
{\rm Gal}(H_+(\ell)/K_+)_{\chi_0^{}} \to 1. $$

Since ${\rm Gal}(H_+(\ell)/K_+)_{\chi_0^{}}$ is, at least, of order $p$, the generating 
$\chi_0^{}$-unit, $\varepsilon_{\chi_0^{}} =: \varepsilon $, is in $E_+(\ell)_{\chi_0^{}}$, 
thus locally a $p$th power at $\ell$, for all $\ell \in {\mathscr L}_p$, 
$\ell \not \equiv 1 \pmod {p^2}$. Thus $\ell$ totally splits in $K(\sqrt[p]{\varepsilon})/K$.
Let $M$ be the compositum $K(\sqrt[p]{\varepsilon})\cdot K_1$, where
$K_1=\Q(\mu_{p^2})$; this Galois field $M$ only depends on $p$ and $\chi_0^{}$
and the primes $\ell \not \equiv 1 \pmod {p^2}$ are inert in $K_1/K$.
Then choose $\ell$ such that the decomposition group 
of ${\mathcal L} \mid \ell$ does not fix $K(\sqrt[p]{\varepsilon})$ (since 
${\rm Gal}(M/K) \simeq (\Z/p\Z)^2$, this allows $p-1$ possibilities).
Thus $\ell$ is inert in $K(\sqrt[p]{\varepsilon})/K$ (contradiction).

\smallskip
Whence the reciprocal. 
\end{proof}

\begin{remark}
This theorem suggests that if the sets ${\mathscr E}_{\ell}(p)$ are random 
when $\ell$ varies and independent, the (conjectural) triviality of $\Cl_+$ is
{\it a consequence} of a natural $p$-adic property of Gauss sums and the
statement does exist with $N=1$.

\smallskip
On the contrary, the structure of $\Cl_-$ is independent of the Gauss 
sums because the {\it even components} ${\rm g}_c(\ell)_{\chi}$ are 
global $p$th powers for all $\ell \in {\mathscr L}_p$ (Remark 
\ref{gooddef}\,(ii)) and do not yield any obstruction ! Thus the cases
of non-triviality of $\Cl_-$ may follow standard probabilities under
the monogenous case.
\end{remark}

\subsection{Study of the case $p=37$}
So it is fundamental to see if the sets ${\mathscr E}_\ell(p)$ 
are independent (or not) of the choice of $\ell \in {\mathscr L}_p$
for ${\mathscr E}_0(p) \ne \ev$.
We analyse the case of $p=37$ 
($n_0=32$) giving $\order \Cl_{\omega^5} = 37$
and compute (in ${\sf expp}$) the sets ${\mathscr E}_\ell(37)$ 
when $\ell \in {\mathscr L}_{37}$ varies. 
If $n_0 \in {\mathscr E}_\ell(37)$, this means that 
${\mathfrak L}_{\chi*}$ is necessarily $37$-principal and then
${\rm g}_c(\ell)_{\omega^5} \in K^{\times 37}$:

\footnotesize
\smallskip
\begin{verbatim}
{p=37;c=lift(znprimroot(p));P=polcyclo(p)+Mod(0,p);X=Mod(x,P);
for(i=1,100,el=1+2*i*p;if(isprime(el)!=1,next);g=znprimroot(el);
print("el=",el," g=",lift(g));J=1;for(i=1,c-1,Ji=0;for(k=1,el-2,kk=znlog(1-g^k,g);
e=lift(Mod(kk+i*k,p));Ji=Ji-X^e);J=J*Ji);LJ=List;Jj=1;for(j=1,p-1,Jj=lift(Jj*J);
listinsert(LJ,Jj,j));for(m=1,(p-3)/2,n=2*m;Sn=Mod(1,P);for(a=1,(p-1)/2,
an=lift(Mod(a,p)^(n-1));Jan=component(LJ,an);sJan=Mod(0,P);
for(j=0,p-2,aj=lift(Mod(a*j,p));sJan=sJan+x^(aj)*component(Jan,1+j));
Sn=Sn*sJan);if(Sn==1,print("     exponent of p-primarity: ",n))))}

el=149      g=2                          el=3331     g=3   expp: 22
el=223      g=3                          el=3701     g=2
el=593      g=3                          el=3923     g=2
el=1259     g=2                          el=4219     g=2   expp: 18,16
el=1481     g=3    expp: 30              el=4441     g=21
el=1777     g=5                          el=4663     g=3
el=1999     g=3                          el=5107     g=2
el=2221     g=2                          el=5477     g=2
el=2591     g=7    expp: 34              el=6143     g=5   expp: 28
el=2887     g=5                          el=6217     g=5
el=3109     g=6                          el=6661     g=6
el=3257     g=3                          el=6883     g=2
-------------------------------------------------------------------------
el=742073   g=3    expp: 12              el=768343   g=11  expp: 18
el=742369   g=7                          el=768491   g=10 
el=742591   g=3                          el=768787   g=2   expp: 20
el=743849   g=3                          el=769231   g=11  expp: 24
el=743923   g=3    expp: 16              el=769453   g=2   expp: 30
el=744071   g=22                         el=772339   g=3
el=744811   g=10                         el=773153   g=3   expp: 14
el=744959   g=13   expp: 10              el=774337   g=5   expp: 28
el=745033   g=10   expp: 16              el=774929   g=3   expp: 18
el=745181   g=2                          el=775669   g=10  expp: 18
el=745477   g=2                          el=776483   g=2
el=745699   g=2                          el=776557   g=2   expp: 20
el=746069   g=2                          el=777001   g=31  expp: 18,28
el=746957   g=2                          el=778111   g=11
el=747401   g=3                          el=778333   g=2   expp: 28
el=747919   g=3                          el=778777   g=5
el=748807   g=6    expp: 22              el=779221   g=2 
el=749843   g=2    expp: 34              el=779591   g=7
el=750287   g=5                          el=779887   g=10  expp: 18
el=750509   g=2    expp: 14,22           el=780257   g=3   expp: 8
el=751027   g=3                          el=780553   g=10
el=751841   g=3    expp: 14,16,24        el=781367   g=5   expp: 34
el=752137   g=10   expp: 8           *el=781589   g=2   expp: 32
el=752359   g=3    expp: 18              el=782107   g=2 
el=752581   g=2    expp: 16              el=782329   g=13  expp: 18
el=752803   g=2    expp: 22,32           el=782921   g=3   expp: 20
el=753617   g=3                          el=783143   g=5
el=753691   g=11   expp: 16              el=783661   g=2
el=753839   g=7    expp: 4,22            el=784327   g=3 
el=754283   g=2                          el=784697   g=3 
el=755171   g=6                          el=784919   g=7 
el=755393   g=3    expp: 22              el=785363   g=2 
el=756281   g=3    expp: 2               el=786251   g=2 
el=756799   g=15   expp: 18              el=786547   g=2
el=757243   g=2                          el=787139   g=2   expp: 20
el=757909   g=2    expp: 16              el=787361   g=6
el=758279   g=7                          el=787879   g=6   expp: 10,18,20
el=758501   g=2    expp: 18              el=788027   g=2   expp: 34
el=759019   g=2                          el=789137   g=3   expp: 24
el=759167   g=5    expp: 12              el=790099   g=2
el=759463   g=3                          el=791209   g=7
el=759833   g=3    expp: 4               el=791431   g=12
el=760129   g=11                         el=791801   g=3 
el=760499   g=2                         *el=792023   g=5   expp: 32
el=762053   g=2                          el=792689   g=3
el=762571   g=10                         el=793207   g=5
el=763237   g=2                          el=795427   g=2
el=764051   g=2                         *el=795649   g=22  expp: 2,32
el=764273   g=3                          el=795797   g=2
el=764717   g=2    expp: 2               el=795871   g=3
el=765383   g=5                          el=796759   g=3 
el=765827   g=2    expp: 34              el=796981   g=7 
el=766049   g=3    expp: 22              el=797647   g=3 
el=766937   g=3    expp: 34              el=797869   g=10
el=767381   g=2    expp: 18              el=798461   g=2
el=767603   g=5    expp: 34              el=798757   g=2 
el=767677   g=5                          el=800089   g=7   expp: 20
\end{verbatim}

\normalsize
\smallskip
For $\ell = 149, 223, 593, 1259, 1777, \ldots\,$, ${\mathscr E}_\ell(37) = \ev$,
which proves the Vandiver conjecture for $p=37$ a great lot of times.
For $\ell = 1481$ one finds a $p$-primarity for 
$\chi^* = \omega^7$ ($\chi = \omega^{30} \ne \omega^{32}$).
Theorem \ref{N} applies at will.

\smallskip
It remains to give statistics about the $p$-principality (or not) 
of the ${\mathfrak L}_{\chi_0^*}$ when $\ell \in {\mathscr L}_p$ varies.
For $p=37$, ${\mathfrak L}_{\chi_0^*}$ is $37$-principal if and only if 
${\mathfrak L}$ is principal since the class number of $K$ is $h = 37$.

\subsubsection{Table of the classes of 
${\mathfrak L}$ for $p=37$.}\label{classes37}

We give a table with a generator of ${\mathfrak L}$ in the principal 
cases (indicated by $*$). Otherwise, the class of ${\mathfrak L}$
is of order $37$ in $K$.
We only write the cases ${\mathscr E}_\ell (37) \ne \ev$:

\footnotesize
\smallskip
\begin{verbatim}
{p=37;c=lift(znprimroot(p));P=polcyclo(p);K=bnfinit(P,1);P=P+Mod(0,p);
X=Mod(x,P);Lsplit=List;N=0;for(i=1,2000,el=1+2*i*p;if(isprime(el)!=1,next);
N=N+1;listinsert(Lsplit,el,N));for(j=1,N,el=component(Lsplit,j);
F=bnfisintnorm(K,el);if(F!=[],print("el=",el," ",component(F,1)));
g=znprimroot(el);J=1;for(i=1,c-1,Ji=0;for(k=1,el-2,kk=znlog(1-g^k,g);
e=lift(Mod(kk+i*k,p));Ji=Ji-X^e);J=J*Ji);LJ=List;
Jj=1;for(j=1,p-1,Jj=lift(Jj*J);listinsert(LJ,Jj,j));for(m=1,(p-3)/2,
n=2*m;Sn=Mod(1,P);for(a=1,(p-1)/2,an=lift(Mod(a,p)^(n-1));
Jan=component(LJ,an);sJan=Mod(0,P);for(j=0,p-2,aj=lift(Mod(a*j,p));
sJan=sJan+x^(aj)*component(Jan,1+j));Sn=Sn*sJan);
if(Sn==1,print("el=",el," expp:",n))))}

el=1481    expp: 30                     el=56167   expp: 10,14,26                            
el=2591    expp: 34                     el=57203   expp: 34                                  
el=3331    expp: 22                     el=58313   expp: 28
el=4219    expp: 16,18                  el=58757   expp: 16,18
el=6143    expp: 28                     el=58831   expp: 24,30
el=7993    expp: 16,20                  el=59497   expp: 28
el=8363    expp: 8                      el=61051   expp: 10
el=9769    expp: 20                     el=62383   expp: 2
el=10657   expp: 4,18,26                el=62753   expp: 2
el=12433   expp: 20                     el=63493   expp: 2
el=13099   expp: 28                     el=64381*  expp: 6,32  [x^20+x^9+x]   
el=14431   expp: 4,14,22                el=66749   expp: 30
el=17021   expp: 6                      el=67489*  expp: 30,32 [x^24-x^3-x^2]     
el=17909   expp: 30                     el=67933   expp: 6
el=18131   expp: 22                     el=68821*  expp: 32    [x^15-x^9+x^4]     
el=19463   expp: 6                      el=69931   expp: 12
el=20129   expp: 6                      el=71411   expp: 4
el=21017   expp: 2,4                    el=72817   expp: 28
el=21313   expp: 18                     el=74149   expp: 2
el=21757   expp: 8                      el=75407   expp: 10
el=22349   expp: 8                      el=75629   expp: 12, 20
el=23459   expp: 6                      el=76961   expp: 14
el=23977   expp: 26                     el=78737   expp: 28
el=25087   expp: 26                     el=79181   expp: 10
el=25457   expp: 30                     el=80513   expp: 16, 26
el=29009   expp: 8,24                   el=81031   expp: 18, 34
el=30859   expp: 2                      el=82067   expp: 34
el=32783*  expp: 32 [x^11+x^3+x]        el=83621   expp: 34
el=33301   expp: 30                     el=83843   expp: 2
el=33967   expp: 26                     el=84731   expp: 6
el=36187   expp: 8                      el=85027   expp: 26
el=37889   expp: 16                     el=86729   expp: 22
el=38629   expp: 22                     el=86951   expp: 8
el=40627   expp: 30                     el=87691   expp: 24
el=40849   expp: 6                      el=91243   expp: 22, 34
el=42773   expp: 4                      el=91909   expp: 30
el=45289   expp: 8                      el=94351   expp: 10
el=45659   expp: 26                     el=94573   expp: 18
el=48619   expp: 8                      el=95239   expp: 18, 28
el=48989   expp: 20                     el=96497   expp: 10
el=51283   expp: 14,16                  el=98347   expp: 28
el=51431   expp: 20                     el=98939   expp: 30
el=53281   expp: 16                     el=99679   expp: 10, 22
el=55057   expp: 20                     el=100049  expp: 14
\end{verbatim}

\normalsize
\smallskip
Give some examples (${\mathfrak L}^{1+s_{-1}}$ is always principal giving
an easy characterization):

\smallskip
(ii) Non-principal case ${\mathfrak L} \mid 149$.
The instruction {$\sf bnfisintnorm(K,149^k)$}:

\footnotesize
\smallskip
\begin{verbatim}
{P=polcyclo(37);K=bnfinit(P,1);for(k=1,2,print(bnfisintnorm(K,149^k)))}
\end{verbatim}

\normalsize
\smallskip
yields an empty set for $k=1$ (since ${\mathfrak L}$ is not principal)
and, for $k=2$, it gives (with $x=\zeta_{37}$) the $18$ conjugates 
of the real integer:

\footnotesize
\smallskip
\begin{verbatim}
 -2*x^35-2*x^34-x^32-2*x^31+x^29-x^28-2*x^27-2*x^24-x^23+x^22-2*x^20-x^19
                     -x^17-2*x^16+x^14-x^13-2*x^12-2*x^9-x^8+x^7-2*x^5-x^4-2*x^2-2*x
\end{verbatim}

\normalsize
\smallskip
(i) Principal case ${\mathfrak L} \mid 32783$.  
The principal ${\mathfrak L}$ are rare 
(which comes from density theorems); the first one is
${\mathfrak L}=(\zeta_{37}^{11} + \zeta_{37}^3 + \zeta_{37})$ where 
$\ell = 32783$. Thus in that case, in the relation
${\mathfrak L}_{\chi_0^*}^{b_c({\chi_0^*})} 
= ({\rm g}_c(\ell)_{\chi_0^*})$, the number
${\rm g}_c(\ell)_{\chi_0^*}$ must be a global $37$th power (which
explains that one shall find the exponent of $37$-primarity 
$n_0=32$ equal to that of $37$-irregularity in the table); 
unfortunately, the data are too large to be given. 

\smallskip
Nevertheless, the reader can easily compute 

${\sf factor(norm(Sn)) = 32783^{37 \cdot 16 \cdot 9}}$ 
and use ${\sf K=bnfinit(P,1)}; {\sf idealfactor(K,Sn)}$,

which gives the $37$th power of ${\mathfrak L} \mid 32783$. 

We obtain the following 
excerpts of the table (up to $10^6$) of principal cases:

\footnotesize
\smallskip
\begin{verbatim}
el=32783    expp:32      el=64381    expp:6,32         el=67489    expp:30,32
el=68821    expp:32      el=108929   expp:32           el=132313   expp:32
el=325379   expp:10,32   el=332039   expp:6,10,14,32   el=351797   expp:32
el=364451   expp:28,32   el=387169   expp:32           el=396937   expp:32
el=960151   expp:32      el=973397   expp:32           el=983239   expp:32
el=1000777  expp:32      el=1002109  expp:2,32         el=1040959  expp:20,32
\end{verbatim}

\normalsize
\subsubsection{Densities of the exponents of $p$-primarity.}\label{vary}
The following program intends to show that all exponents of $p$-primarity
are obtained, with (perhaps) some specific densities, taking sufficientely 
many $\ell \in {\mathscr L}_p$
(each even $n \in [2, p-3]$, such that ${\rm g}_c(\ell)_{\omega^{p-n}}$
is $p$-primary for some new $\ell$, is counted in the $(n/2)$th 
component of the list ${\sf Eel}$). 

\footnotesize
\smallskip
\begin{verbatim}
{p=37;c=lift(znprimroot(p));P=polcyclo(p)+Mod(0,p);X=Mod(x,P);Nel=0;Npp=0;Eel=List;
for(j=1,(p-3)/2,listput(Eel,0,j));for(i=1,1000,el=1+2*i*p;if(isprime(el)!=1,next);
g=znprimroot(el);Nel=Nel+1;J=1;for(i=1,c-1,Ji=0;for(k=1,el-2,kk=znlog(1-g^k,g);
e=lift(Mod(kk+i*k,p));Ji=Ji-X^e);J=J*Ji);LJ=List;Jj=1;for(j=1,p-1,Jj=lift(Jj*J);
listinsert(LJ,Jj,j));for(m=1,(p-3)/2,n=2*m;Sn=Mod(1,P);for(a=1,(p-1)/2,
an=lift(Mod(a,p)^(n-1));Jan=component(LJ,an);sJan=Mod(0,P);for(j=0,p-2,
aj=lift(Mod(a*j,p));sJan=sJan+x^(aj)*component(Jan,1+j));Sn=Sn*sJan);
if(Sn==1,Npp=Npp+1;listput(Eel,1+component(Eel,n/2),n/2);
print(Nel," ",Npp," ",el," ",Eel))))}
\end{verbatim}

\normalsize
\smallskip
In the first column, one shall find the index $i$ (in ${\sf Nel}$) of the prime 
$\ell_i$ considered; if some index $i$ is missing, this means that
${\mathscr E}_{\ell_i} (p) = \ev$. The second integer gives the whole 
number of exponents of $p$-primarity obtained at this step (in ${\sf Npp}$); 
then the third one is $\ell_i$ (in ${\sf el}$).
In some cases, a prime $\ell$ gives rise to several 
exponents of $p$-primarity

\smallskip
(i) Results for $p=37$. The end of the table for the selected interval is:

\footnotesize
\smallskip
\begin{verbatim}
Nel  Npp    el 
3015 1426 1414067  [83,95,84,91,80,80,86,83,92,83,97,76,83,78,85,74,76]
3015 1427 1414067  [83,95,84,91,80,80,86,83,92,83,97,76,83,78,86,74,76]
3027 1428 1419839  [83,95,84,91,80,80,86,83,92,83,98,76,83,78,86,74,76]
3030 1429 1420949  [83,95,84,91,80,80,86,83,92,83,98,76,83,78,86,75,76]
3032 1430 1421911  [83,95,85,91,80,80,86,83,92,83,98,76,83,78,86,75,76]
3033 1431 1422133  [83,95,86,91,80,80,86,83,92,83,98,76,83,78,86,75,76]
3042 1432 1428127  [83,96,86,91,80,80,86,83,92,83,98,76,83,78,86,75,76]
3889 1819 1863913[106,114,108,113,99,111,115,100,117,113,116,93,103,97,108,103,103]
3894 1820 1865911[106,114,108,114,99,111,115,100,117,113,116,93,103,97,108,103,103]
3898 1821 1868501[106,114,108,114,100,111,115,100,117,113,116,93,103,97,108,103,103]
3900 1822 1869389[106,114,108,114,100,112,115,100,117,113,116,93,103,97,108,103,103]
3900 1823 1869389[106,114,108,114,100,112,115,101,117,113,116,93,103,97,108,103,103]
3900 1824 1869389[106,114,108,114,100,112,115,101,117,113,116,93,104,97,108,103,103]
\end{verbatim}

\normalsize
The penultimate column corresponds to the exponent of 
$37$-irregularity $n_0=32$; since there is no counterexamples
to Vandiver's conjecture, when this component increases,
this means that the new $\ell$ gives rise to a principal ${\mathfrak L}$
for which ${\rm g}_c(\ell)_{\omega^5}$ is a $37$th power.

\smallskip
(ii) Results for $p=157$.
For $p=157$ (exponents of $p$-irregularity $62, 110$),
one finds the partial analogous information after 
$590$ distinct primes $\ell \in {\mathscr L}_p$ tested  
(proving also Vandiver's conjecture for a lot of times):

\footnotesize
\smallskip
\begin{verbatim}
Nel  Npp    el 
590 309 1161487 [9,3,2,6,8,3,1,4,5,10,3,1,3,1,6,3,4,4,2,2,1,2,5,
                5,3,2,2,1,5,7,6,2,2,1,5,5,5,4,4,3,3,4,5,4,5,5,5,5,5,3,6,1,6,3,5,4,5,
                                            0,2,3,5,7,3,3,3,2,4,4,7,6,6,5,6,1,7,4,7]
590 310 1161487 [9,3,2,6,8,3,1,4,5,10,3,1,3,1,6,3,4,4,2,2,1,2,5,
                5,3,2,2,1,5,7,6,2,2,1,5,5,5,4,4,3,3,4,5,4,5,6,5,5,5,3,6,1,6,3,5,4,5,
                                            0,2,3,5,7,3,3,3,2,4,4,7,6,6,5,6,1,7,4,7]
590 311 1161487 [9,3,2,6,8,3,1,4,5,10,3,1,3,1,6,3,4,4,2,2,1,2,5,
                5,3,2,2,1,5,7,6,2,2,1,5,5,5,4,4,3,3,4,5,4,5,6,5,5,5,3,6,1,6,3,5,4,5,
                                            0,2,3,5,7,3,3,3,2,4,5,7,6,6,5,6,1,7,4,7]
\end{verbatim}

\normalsize
\smallskip
The remaining column of zeros (for $n/2=58$) stops at the following lines:

\footnotesize
\smallskip
\begin{verbatim}
602 318 1185979 [9,3,2,6,8,3,2,4,6,10,3,1,3,1,6,4,4,4,2,2,1,2,5,
                5,3,2,2,1,5,7,6,3,2,1,5,5,5,4,4,3,3,4,5,4,5,6,5,5,5,3,6,1,6,4,5,4,6,
                                            0,2,3,5,7,3,3,3,3,4,5,7,6,6,5,6,1,7,4,7]
602 319  1185979 [9,3,2,6,8,3,2,4,6,10,3,1,3,1,6,4,4,4,2,2,1,2,5,
                5,3,2,2,1,5,7,6,3,2,1,5,5,5,4,4,3,3,4,5,4,5,6,5,5,5,3,6,1,6,4,5,4,6,
                                            1,2,3,5,7,3,3,3,3,4,5,7,6,6,5,6,1,7,4,7]
602 320 1185979 [9,3,2,6,8,3,2,4,6,10,3,1,3,1,6,4,4,4,2,2,1,2,5,
                5,3,2,2,1,5,7,6,3,2,1,5,5,5,4,4,3,3,4,5,4,5,6,5,5,5,3,6,1,6,4,5,4,6,
                                            1,2,4,5,7,3,3,3,3,4,5,7,6,6,5,6,1,7,4,7]
\end{verbatim}

\normalsize
These numbers may depend on the orders of $\omega^n$ and/or 
$\omega^{p-n}$, but this needs to be clarified taking much 
$\ell \in {\mathscr L}_p$.

\subsubsection{Vandiver's conjecture and $p$-adic regulator of $K_+$.}
\label{regul}
We return to the case $p=37$ and $n_0=32$.
We see that $\omega^{32}$ 
is a character of order $9$, hence a character of the real subfield $k_9$
of degree $9$, which is such that ${\mathcal T}_{k_9} \ne 1$ 
from the reflection relation \eqref{spiegel}; 
so, $k_9$ admits a cyclic $37$-ramified extension of degree $37$
which is not unramified. To verify, we use
\cite[Program I]{Gr2}, for real fields, which indeed gives 
$\order {\mathcal T}_{k_9} = 37$ (${\sf nt}$ must verify
${\sf p^{nt}} > p^{t}$, the exponent of ${\mathcal T}$):

\footnotesize
\smallskip
\begin{verbatim}
{p=37;n=32;d=(p-1)/gcd(p-1,n);P=polsubcyclo(p,d);K=bnfinit(P,1);nt=6;
Kpn=bnrinit(K,p^nt);Hpn=component(component(Kpn,5),2);L=List;
e=component(matsize(Hpn),2);R=0;for(k=1,e-1,c=component(Hpn,e-k+1);
if(Mod(c,p)==0,R=R+1;listinsert(L,p^valuation(c,p),1)));if(R>0,
print("rk(T)=",R,"  K is not ",p,"-rational ",L));
if(R==0,print("rk(T)=",R,"  K is ",p,"-rational"))}

rk(T)=1   K is not 37-rational    List([37])
\end{verbatim}

\normalsize
\smallskip
We find here another interpretation of the reflection theorem since
we have the typical formula
$\order {\mathcal T}_+ = \order \Cl_+ \cdot  \order {\mathcal R}_+$,
where the $p$-group ${\mathcal R}_+$ is the normalized 
$p$-adic regulator of $K_+$ \cite[Proposition 5.2]{Gr5}.
Whence $\order {\mathcal T}_\chi = \order \Cl_ \chi \cdot 
\order {\mathcal R}_\chi$ and ${\mathcal R}_{\chi*}=1$, for all 
$\chi \in {\mathscr X}_+$; but we have 
$\order {\mathcal T}_{\chi*} = \order \wt \Cl_{\chi*}$ 
for the subgroup $\wt \Cl_{\chi*}$ of $\Cl_{\chi*}$.
The above data shows that the relation $\order {\mathcal T}_{\chi_0^{}} = 37$ 
comes from $\order {\mathcal R}_{\chi_0^{}}=37$, which is not surprising:

\begin{remark}
We have the analytic formula $\order \Cl_{\chi_0^{}} = 
\order (E_{\chi_0^{}} / \langle \eta_{\chi_0^{}} \rangle)$, where $\eta$
is a suitable cyclotomic unit; so a classical method (explained
in \cite[Corollary 8.19]{Wa}, applied in \cite{BH,HHO} 
and developped in \cite{T1,T2}) 
consists in finding $\ell \in {\mathscr L}_p$ such that 
$\eta_{\chi_0^{}}$ is not a local $p$th power at $\ell$ proving 
Vandiver's conjecture at $\chi_0^{}$; so when we find that 
${\mathcal R}_{\chi_0^{}} \ne 1$ (with $\Cl_{\chi_0^{}}=1$), this
means that $\eta_{\chi_0^{}}$ generates $E_{\chi_0^{}}$ and is 
a local $p$th power at $p$ by $p$-primarity, 
so that $K\big (\sqrt[p]{\eta_{\chi_0}} \big)$ is contained in the 
$\chi_0^*$-component of the $p$-Hilbert class field of $K$. 

We shall give in \S\,\ref{resymbol} some insights in this direction
to state new heuristics for the probability of $p$-primarity of
${\rm g}_c(\ell)_{\chi_0^*}$ to be at most $\frac{O(1)}{p^2}$.
\end{remark}

\section{Heuristics -- Probability of a counterexample}

\subsection{Use of classical standard probabilities}
We may suppose in a first approximation 
that, for a given $p$, the sets ${\mathscr E}_\ell(p)$
of exponents of $p$-primarity of primes $\ell \in {\mathscr L}_p$, 
are random with the same behavior as for the 
set ${\mathscr E}_0(p)$ of exponents of $p$-irregularity.
More precisely, assume, as in Washington's book (see in
\cite{Wa}, the Theorem 5.17 and some statistical 
computations), that for given primes $p$ and $\ell \in {\mathscr L}_p$, 
the probabilities of a cardinality $k$ is
$\binom{N} {j} \cdot \big(1-\frac{1}{p} \big)^{N-j} \cdot \big(\frac{1}{p} \big)^j$
(where $N:= \frac{p-3}{2}$).
This would imply that, for $p$ given,
${\mathscr E}_\ell(p) \ne \ev$ for many $\ell \in {\mathscr L}_p$, but that
${\mathscr E}_\ell(p) = \ev$ in a proportion close to $e^{- \frac{1}{2}}$,
which is in accordence with previous tables.
Then the probability, for $p$ and $\ell$ given, of 
${\mathscr E}_0(p) \cap {\mathscr E}_\ell(p) \ne \ev$
with cardinalities $j \in [0, N]$ and $k \in [0, N]$ fixed, is:
$$1 - \frac{(N-k) ! \cdot (N-j) !}{N ! \cdot (N-k-j) !}. $$

So, an approximation of the whole probability of 
${\mathscr E}_0(p) \cap {\mathscr E}_\ell(p) \ne \ev$ is:
\begin{equation}\label{stupid}
\sm_{j, \, k \geq 0} 
\binom{N} {j} \binom{N} {k} \cdot \Big(1-\frac{1}{p} \Big)^{2\,N-j-k} 
\!\cdot \Big(\frac{1}{p} \Big)^{j+k}\! \cdot 
 \Big( 1 - \frac{(N-k) ! \cdot (N-j) !}{N ! \cdot (N-k-j) !}  \Big). 
\end{equation}

The computations show that this expression is around $\frac{1}{2\,p}$,
which does not allow to conclude easily for a single $\ell$, but {\it this 
does not take into account} the ``infiniteness'' of ${\mathscr L}_p$ giving, 
a priori, {\it independent informations}, but limited by the Theorem \ref{cyclicity} 
on periodicities due to the density theorem (see the Weil interpretation
of Jacobi sums defining Hecke Gr\"ossencharacters \cite[Theorem, p. 489]{We} 
where a module of definition of our Jacobi sums is $p^2$, which may give 
an order of magnitude of the cardinality of this ``infiniteness''). So
this must be put in relation with the Theorem \ref{thmp} to
characterize ``non-Vandiver''.

\subsection{New heuristics and probabilities}\label{new}
Many reasons imply that the generic probability $\frac{1}{p}$
must be replaced by a much lower one:

\subsubsection{Results from ${\rm K}$-theory.}
For some characters $\chi \in {\mathscr X}_+$,
of the form $\chi =: \omega^{p-(1+h)}$, for small 
$h=2, 4,\ldots\,$, one may prove 
that $\Cl_{\omega^{p-(1+h)}} = 1$, as the case of $\Cl_{\omega^{p-3}} = 1$
proved unconditionally by Kurihara \cite{K}; then Soul\'e proved in \cite{So1} that for 
$n \in [2, p-3]$ even, $\Cl_{\omega^{p-n}} = 1$ for all $p$ large 
enough (see also \cite{Gh, So2, BEH} among other references 
applying the same approach via ${\rm K}$-theory). 
Unfortunately these bounds are not usable in practice, 
but demonstrate the existence of a fundamental general principle.

\subsubsection{Archimedean aspects.}
At the opposite, for $\chi \in {\mathscr X}_+$ of 
small order, $\Cl_{\chi}$ may be trivial because of the 
``archimedean'' order of magnitude of the whole class number of
the subfield of $K_+$ fixed by ${\rm Ker}(\chi)$
(which is proved for the quadratic case when $p\equiv 1 \pmod 4$, the 
cubic case  when $p\equiv 1 \pmod 3$); see the tables of Schoof \cite{Sc} 
for serious arguments about the order of magitude of the whole class number.
Moreover, we have the $p$-rank
$\epsilon$-conjecture for $p$-class groups \cite{EV} that we 
state for the real abelian fields $k_d$ {\it of constant degree} $d$, 
of discriminant $D=p^{d-1}$, when $p \equiv 1\!\! \pmod d$ increases:

\medskip
\centerline{{\it For all $\epsilon>0$ there exists $C_{p,\epsilon}$ 
such that} ${\rm log}(\order( \Cl_{k_d} / \Cl_{k_d}^p)) \leq {\rm log}(C_{p,\epsilon})
+\epsilon  \cdot {\rm log}(p)$,}

\smallskip
which would give $\Cl_{k_d}=1$ for ${\rm log}(p) > 
\frac{{\rm log}(C_{p,\epsilon})}{1 - \epsilon}$ and any $\epsilon <1$.
But this does not apply for any $p$ with ``small'' $d$ and the 
constant $C_{p,\epsilon}$ is not effective.

\subsubsection{Heuristics about Gauss sums.}
The standard probabilities \eqref{stupid} assume that when 
$\ell \in {\mathscr L}_p$ varies, the sets ${\mathscr E}_\ell(p)$ are 
{\it random and independent}, giving probabilities close to~$0$,
which is not the case when $p$ is irregular at some 
$\chi_0^* = \omega^{p-n_0}$ with $\Cl_{\chi_0^*} \simeq \Z/ p^e \Z$, $e \geq 1$,
and when ${\rm g}_c(\ell)_{\chi_0^*}$ is a global $p$th power because 
${\mathfrak L}_{\chi_0^*}^{p^{e-1}}$ is $p$-principal.

\smallskip
Fix $\ell \in {\mathscr L}_p$ such that ${\mathfrak L}_{\chi_0^*}$ generates
$\Cl_{\chi_0^*} \simeq \Z/p^e\Z$ (thus ${\rm g}_c(\ell)_{\chi_0^*}$ is not 
a global $p$th power); put (Proposition \ref{varpi})
${\rm g}_c(\ell)_{\chi_0^*} = 1+ \beta_0(\ell) \cdot \varpi^{p-n_0}, \ 
\beta_0(\ell) \in \Z_p[\varpi]$,
where $\beta_0(\ell)$ is invertible modulo $\varpi $ if and only if 
${\rm g}_c(\ell)_{\chi_0^*}$ is not $p$-primary. 

\smallskip
Whatever $\ell' \in {\mathscr L}_p$ and ${\mathfrak L}' \mid \ell'$, 
one has, from \S\,\ref{ellell'}\,(ii),
${\rm g}_c(\ell')_{\chi_0^*} \equiv {\rm g}_c(\ell)_{\chi_0^*}^r 
\pmod p$, with $r \in [0, p^e-1]$ ($r=0$ if ${\mathfrak L}'_{\chi_0^*}$ 
is $p$-principal, thus ${\rm g}_c(\ell')_{\chi_0^*} \in K^{\times p}$), 
giving:
\begin{equation}\label{*}
{\rm g}_c(\ell')_{\chi_0^*} =: 1+ \beta_0(\ell')
\cdot \varpi^{p-n_0}, \ \, \beta_0(\ell') \equiv  r \cdot \beta_0(\ell) \!\! 
\pmod{\varpi}.
\end{equation}

Contrary to the classical idea that $\beta_0(\ell) \pmod \varpi $ 
follow standard probabilities $\frac{O(1)}{p}$
(even under the condition ${\rm g}_c(\ell)_{\chi_0^*} \notin K^{\times p}$),
we propose the following heuristic:

\smallskip
\begin{pushright}
{\it For each $\chi \in {\mathscr X}_+$, the mod $p$ values, at 
$\chi^* = \omega\,\chi^{-1}$, of the Gauss sums (more precisely 
of the $\psi^{-c}(c) \cdot {\rm g}_c(\ell) = J_1 \cdots J_{c-1}$), 
are uniformly distributed (or at least with explicit non-trivial 
densities), when $\ell \in {\mathscr L}_p$ varies.}
\end{pushright}

\smallskip
Because of the density theorems on the ideal classes when $\ell$
varies in ${\mathscr L}_p$ and $\chi$ in ${\mathscr X}_+$, we must 
examine two cases concerning the $\chi$-components of $\Cl$ when 
there exists $\chi_0^{} = \omega^{n_0} \in {\mathscr X}_+$ such that 
$\Cl_{\chi_0^*} \simeq \Z/p^e\Z$, $e \geq 1$:

\smallskip
\quad (a) $\chi \ne \chi_0^{}$ and $\Cl_{\chi*} = 1$.
The numerical experiments show that when $\ell \in {\mathscr L}_p$ varies, 
${\rm g}_c(\ell)_{\chi*} = 1+ \beta(\ell) \cdot \varpi^{p-n} $, 
with random $\beta(\ell) \pmod \varpi$ (probabilities $\frac{O(1)}{p}$).

\smallskip
\quad (b) $\chi = \chi_0^{}$ (and $\Cl_{\chi_0^*}\ne 1$).
If ${\rm g}_c(\ell)_{\chi_0^*}$ is $p$-primary for 
some given generator ${\mathfrak L}_{\chi_0^*}$, 
then from \eqref{*} all the ${\rm g}_c(\ell')_{\chi_0^*}$ 
are $p$-primary, whatever the class
of ${\mathfrak L}'_{\chi_0^*}$ ($p^e$ possibilities) because
$\beta_0(\ell') \equiv 0 \pmod \varpi$. So,
$n_0$ is always in ${\mathscr E}_\ell(p)$ and 
${\mathscr E}_0(p) \cap {\mathscr E}_\ell(p) \ne \ev$
for all $\ell \in {\mathscr L}_p$, which corresponds to
$\Cl_{\chi_0^{}} \ne 1$ and the 
non-cyclicity of $\Cl_{\chi_0^*}^{(p)}$ (Theorem \ref{cyclicity}).

\smallskip
Thus, to have analogous densities of $p$-primarity 
on ${\mathscr L}_p$ (as for the $p$-principal case (a)), 
$\beta_0(\ell) \equiv 0 \pmod \varpi$ ({\it under the condition 
${\rm g}_c(\ell)_{\chi_0^*} \notin K^{\times p}$}) must occur at least
$p$ times less, giving a probability in $\frac{O(1)}{p^2}$ 
instead of $\frac{O(1)}{p}$; it is even possible 
that such a circumstance be of probability $0$
depending on more accurate properties of Gauss or 
Jacobi sums; for this, the computation of 
$\beta(\ell)$ should be very interesting (see \cite{T2} where,
for any $\ell \equiv 1 \pmod p$, the coefficients $d_{i,k}$ of 
$J_i := \sum_{k=0}^{p-1} d_{i,k}\,\zeta_p^k$,
with $\sum_{k=0}^{p-1} d_{i,k}=1$, are studied and 
the starting point of future investigations, in relation with 
the other heuristics).

\subsubsection{Use of $p$th power residue symbols and cyclotomic units.}
\label{resymbol}
We refer to \cite[\S\,8.3]{Wa} for the classical $p$-adic interpretation of the
numbers $\order \Cl_{\chi}$, for $\chi \in  {\mathscr X}_+$,
as indices of the form $(E_\chi : F_\chi)$, where $F$ is the 
group of cyclotomic units. 

\smallskip
We need the following $p$th power criterion (from \cite[II.6.3.8]{Gr1}):

\begin{lemma} \label{ppower}
Let $\alpha \in K^\times$ be a pseudo-unit (namely, $\alpha$
is prime to $p$ and $(\alpha)={\mathfrak a}^p$).
Let any set ${\mathscr S}$ of places ${\mathfrak q}$ 
of $K$ such that $\langle \cl({\mathscr S}) \rangle_\Z = \Cl$
(or such that $\langle \cl({\mathscr S}) \rangle_{\Z[G]} = \Cl$ if
$K(\sqrt[p]{\alpha})/\Q$ is Galois).

Then $\alpha\in K^{\times p}$ if and only if $\alpha$ is 
$p$-primary and locally a $p$th power at ${\mathscr S}$ 
(i.e., $\alpha \in K_{\mathfrak q}^{\times p}$ 
for all ${\mathfrak q} \in {\mathscr S}$
where $K_{\mathfrak q}$ is the ${\mathfrak q}$-completion of $K$). 
\end{lemma}

\begin{proof} Consider the non-trivial direction, in the Galois case,
assuming that $\alpha$ is $p$-primary and such that 
$\alpha \in K_{\mathfrak q}^{\times p}$ for all ${\mathfrak q} \in {\mathscr S}$.
So $K(\sqrt[p]{\alpha})/K$ is unramified and ${\mathscr S}$-split; 
thus, due to the Galois condition, all the conjugates of 
${\mathfrak q} \in {\mathscr S}$ split and the Galois group of 
$K(\sqrt[p]{\alpha})/K$ corresponds, by class field theory, to a quotient 
of $\Cl/\langle \cl({\mathscr S}) \rangle_{\Z[G]}$, trivial by assumption. 
\end{proof}

\begin{theorem} Let $\chi_0^{} = \omega^{n_0} \in {\mathscr X}_+$
with $n_0 \in {\mathscr E}_0(p)$ and $\Cl_{\chi_0^*} 
\simeq \Z/p^e\Z$, $e\geq 1$ (i.e., $b_c(\chi_0^*) \sim p^e$). Let
$\eta := \zeta_p^{\frac{1-c}{2}} \frac{1-\zeta_p^c}{1-\zeta_p}$
be a generating real cyclotomic unit, where $c$ is a primitive root 
modulo $p$ (cf. \cite[Propo\-sition 8.11]{Wa}).

\smallskip
\quad (i) There exists an infinite subset 
${\mathscr L}_p(\chi_0^{}) \subseteq  {\mathscr L}_p$ of primes
$\ell$ such that the $G$-module generated by the $p$-class of 
${\mathfrak L} \mid \ell$ is $\Cl_{\chi_0^{}} \oplus \Cl_{\chi_0^*}$. 

\smallskip
\quad (ii)  $\Cl_{\chi_0^{}} \ne 1$ if and only if 
${\rm g}_c(\ell)_{\chi_0^*}$  is locally a $p$th power at ${\mathfrak p}$
but not at ${\mathfrak L}$ ($\ell \in {\mathscr L}_p(\chi_0^{})$).

\smallskip
\quad (iii) $\Cl_{\chi_0^{}} \ne 1$ if and only if 
$\eta_{\chi_0^{}}$ is locally a $p$th power at ${\mathfrak p}$
and at ${\mathfrak L}$ ($\ell \in {\mathscr L}_p(\chi_0^{})$).
\end{theorem}

\begin{proof} (i) In the $\Z_p[G]$-monogenous case,
the ideals ${\mathfrak L}$ are of the form 
$(z) \cdot {\mathfrak A} \cdot {\mathfrak A}^*$, $z \in K^\times$, 
where $\cl({\mathfrak A})$ generates $\Cl_{\chi_0^{}}$ and 
$\cl({\mathfrak A}^*)$ generates $\Cl_{\chi_0^*}$.\,\footnote{If, for instance,
$\Cl_{\chi_0^{}} \simeq \Cl_{\chi_0^*} \simeq  \Z/p\Z$, 
these prime ideals ${\mathfrak L}$ have density $\Big (1-\frac{1}{p} \Big)^2$; 
otherwise, if $\Cl_{\chi_0^{}}=1$ and $\Cl_{\chi_0^*} \simeq  \Z/p\Z$,
the density is $1-\frac{1}{p}$.}

\smallskip
(ii) $\&$ (iii) Define the $p$th power residue symbol 
$\ds \bigg(\frac{\alpha}{{\mathfrak L}} \bigg) := 
\alpha^{\frac{\ell-1}{p}} \!\! \pmod {\mathfrak L}$ for 
${\mathfrak L} \mid \ell \in {\mathscr L}_p$ when $\alpha \in K^\times$
is prime to ${\mathfrak L}$. By abuse of notation, we shall write 
$\ds \bigg(\frac{\alpha}{p} \bigg) = 1$ if $\alpha$ is $p$-primary and
$\ds \bigg(\frac{\alpha}{{\mathfrak L}} \bigg) = 1$ if 
$\alpha \in K_{\mathfrak L}^{\times p}$ is not prime to ${\mathfrak L}$. 
Take $\ell \in {\mathscr L}_p(\chi_0^{})$:

\medskip
\quad (a) Consider $\alpha={\rm g}_c(\ell)_{\chi_0^*}$, where 
$\big ({\rm g}_c(\ell)_{\chi_0^*} \big)
= {\mathfrak L}_{\chi_0^*}^{b_c(\chi_0^*)}$. This gives rise 
to a counterexample to Vandiver's conjecture at $\chi_0^{}$ if and only if 
$\alpha$ is $p$-primary since $\cl({\mathfrak L}_{\chi_0^*})$
is a generator of $\Cl_{\chi_0^*}$; 
it follows that $\ds \bigg(\frac{\alpha}{{\mathfrak L}} \bigg) \ne 1$, 
otherwise, from Lemma \ref{ppower} applied in $H_{\chi_0^{}}$, 
$\alpha={\rm g}_c(\ell)_{\chi_0^*}$ 
should be a global $p$th power (contradiction).

\medskip
\quad (b) Consider $\alpha=\eta_{\chi_0^{}}$. Thus
$b_c(\chi_0^*) \equiv 0 \pmod p$ is equivalent to the $p$-primarity
of $\eta_{\chi_0^{}}$; so a counterexample to Vandiver's conjecture
at $\chi_0^{}$, equivalent to $\eta_{\chi_0^{}} \in E_{\chi_0^{}}^p$, is
equivalent to $\ds \bigg(\frac{\eta_{\chi_0^{}}}{{\mathfrak L}} \bigg) = 1$ since
$\ds \bigg(\frac{\eta_{\chi_0^{}}}{p} \bigg) = 1$.
Whence, with a prime ${\mathfrak L} \mid \ell \in {\mathscr L}_p(\chi_0^{})$:

\smallskip
$\ds \hspace{0.6cm} \Cl_{\chi_0^{}} \ne 1 \Leftrightarrow   
\bigg(\frac{{\rm g}_c(\ell)_{\chi_0^*}}{{\mathfrak L}} \bigg) \ne 1
  \ \&  \  \bigg(\frac{{\rm g}_c(\ell)_{\chi_0^*}}{p} \bigg) = 1
 \Leftrightarrow  \bigg(\frac{\eta_{\chi_0^{}}}{{\mathfrak L}} \bigg) = 
 \bigg(\frac{\eta_{\chi_0^{}}}{p} \bigg) = 1$.
\end{proof}

Let $\chi \in {\mathscr X}_+$
and $\ell \in {\mathscr L}_p(\chi)$ fixed.
If ${\rm Prob} \big(\big(\frac{{\rm g}_c(\ell)_{\chi^*}}
{{\mathfrak L}} \big) \ne 1 \big)$ is close to $1$, 
this suggests a probability around $\frac{O(1)}{p^2}$
for the $p$-primarity of ${\rm g}_c(\ell)_{\chi^*}$
if the two above symbols of $\eta_{\chi}$ are independent
with probabilities $\frac{O(1)}{p}$ for a single $\ell$.

So it is necessary to compute the symbol 
$\bigg(\frac{{\rm g}_c(\ell)_{\chi^*_0}}{{\mathfrak L}} \bigg)$
since ${\rm g}_c(\ell)_{\chi^*_0}$ and ${\mathfrak L}$ are
non-independent data.
For $\chi_0^{}=\omega^{n_0}$, $n_0 \in {\mathscr E}_0(p)$,
the primes $\ell$ of the theorem are not effective, 
but experiments with random $\ell$ seem sufficient for statistics.
Then a first condition for
$\bigg(\frac{{\rm g}_c(\ell)_{\chi_0^*}}{{\mathfrak L}} \bigg)=1$
is that ${\rm g}_c(\ell)_{\chi_0^*}$ be the $p$th power
of an $\ell$-ideal, which is fulfilled since $b_c(\chi_0^*) \equiv 0 \pmod p$.

\smallskip
Then, using the general program computing ${\rm g}_c(\ell)_{\chi_0^*}$
in ${\sf Sn} \in \Z[\zeta_p]$ (in other words {\it not reduced modulo $p$}), 
we divide this integer by the maximal power $\ell^{\, v}$, so that there exists a 
prime ideal ${\mathfrak L} \mid \ell$ which does not divide this 
new integer (still denoted ${\sf Sn}$ and $p$th power of an $\ell$-ideal);
the computation reduces to $R$, prime to ${\mathfrak L}$ and most likely
random, whose symbol $\Big(\frac{R}{{\mathfrak L}}\Big) 
= R^{\frac{\ell-1}{p}} \pmod {{\mathfrak L}}$, computed in ${\sf u}$, 
is immediate and gives the statistics:

\footnotesize
\smallskip
\begin{verbatim}
{p=37;n=32;print("p=",p," n=",n);c=lift(znprimroot(p));P=polcyclo(p);X=Mod(x,P);
for(i=1,100,el=1+2*i*p;if(isprime(el)!=1,next);g=znprimroot(el);M=(el-1)/p;J=1;
for(i=1,c-1,Ji=0;for(k=1,el-2,kk=znlog(1-g^k,g);e=lift(Mod(kk+i*k,p));
Ji=Ji-X^e);J=J*Ji);LJ=List;Jj=1;for(j=1,p-1,Jj=lift(Jj*J);listinsert(LJ,Jj,j));
Sn=1;for(a=1,p-1,an=lift(Mod(a,p)^(n-1));Jan=component(LJ,an);sJan=Mod(0,P);
for(j=0,p-2,aj=lift(Mod(a*j,p));sJan=sJan+X^(aj)*component(Jan,1+j));Sn=Sn*sJan);
Sn=lift(Sn);s=valuation(Sn-1,p);v=valuation(Sn,el);Sn=Sn/el^v;ro=g^M;
for(b=1,p-1,r=lift(ro^b);R=0;for(k=0,p-2,R=R+component(Sn,k+1)*r^k);
if(valuation(R,el)==0,y=R;break));u=lift(Mod(y,el)^M);
print("p=",p," el=",el," v=",v," u=",u);if(s!=0,print("Sn local pth power at P"));
if(Mod(v,p)==0 & u==1,print("Sn local pth power at L"));
if(Mod(v,p)!=0 || u!=1,print("Sn NON local pth power at L"));
if(Mod(v,p)==0 & u==1 & s!=0,print("Sn GLOBAL pth power")))}
p=37  n=32 
el=149     v=259   u=102    Sn NON local pth power at L
el=223     v=259   u=132    Sn NON local pth power at L
el=6883    v=259   u=6850   Sn NON local pth power at L
el=7253    v=259   u=4947   Sn NON local pth power at L
el=32783   v=259   u=1      Sn local pth power at P
el=32783   v=259   u=1      Sn local pth power at L
el=32783   v=259   u=1      Sn GLOBAL pth power
\end{verbatim}

\normalsize
We found $u= \Big(\frac{{\rm g}_c(\ell)_{\chi_0^*}}
{{\mathfrak L}} \Big) =1$ for the following $\ell$
(including the underlined numbers cor\-responding to primes
 $\ell \notin {\mathscr L}_p(\chi_0^{})$ such that 
 ${\rm g}_c(\ell)_{\chi_0^*} \in K^{\times p}$, i.e., ${\mathcal L}$ 
 $p$-principal):

\smallskip
\footnotesize
$\ell \  \in \ $ $\{22571$;$\ul {32783}$;$46103$;$53503$;$57943$;
$\ul {64381}$;$\ul {67489}$;$\ul {68821}$;$79847$;$83177$;$96497$;
$98939$;$104933$;$\ul {108929}$;

$\ \ \ \  117883$;$\ul {132313}$;$146521$;$\ul {146891}$;$151553$;$151849$;$158657$;
$158731$;$\ul {167759}$;
$\ul {172717}$;$197359 $;$\ul {198839}$,$\ul {207497}$$\}$ 

\smallskip
\normalsize
confirming existence and rarity of primes $\ell$ in the interval
$[149;207497]$ such that $u=1$ by accident 
(${\rm g}_c(\ell)_{\chi_0^*} \notin K^{\times p}$, i.e.,
${\mathcal L}$ non-$p$-principal). 

\smallskip
For $n=22 \notin {\mathscr E}_0(37)$, we found $u=1$ 
for the few examples (up to $2 \cdot 10^5$):

\smallskip
\footnotesize
\centerline{$\ell \in \{$$2221$; $2887$; $3923$; $49211$; $51283$; 
$69709$; $147779$; $164503$; $170497$; $179969$;
$192697$; $197803$$\}$}

\smallskip
\normalsize
but ${\rm g}_c(\ell)_{\chi^*}$ 
is not the $p$th power of an $\ell$-ideal, whence it is never in 
$K_{\mathfrak L}^{\times p}$. One finds an exponent of 
$p$-primarity $22$ for $\ell=3331$, then $14$, $16$ for 
$\ell = 51283$, $10$ for $\ell= 147779 $, and $28$ for 
$\ell= 164503$. In the exceptional case $\ell = 3331$, 
${\rm g}_c(\ell)_{\chi^*}$ is $p$-primary. 

This confirms the
expected properties of the symbol 
$\Big(\frac{{\rm g}_c(\ell)_{\chi^*_0}}{{\mathfrak L}} \Big)$.
A similar program computing the two symbols of $\eta_{\chi_0^{}}$ 
gives all expected results.

\subsubsection{Classical heuristics on class groups.}
A first important reason for a very rare occurrence of the non-cyclic 
case for $\Cl_{\chi*}^{(p)}$ may come from classical heuristics on 
$p$-class groups, assuming that they can be applied to ray class 
groups as $\Cl_{\chi*}^{(p)}$ when it is, for instance, of order $p^2$. 

\smallskip
Whatever the (numerous) references concerning this subject and 
independently of some improvements or questions on the relevance
of the formulas giving ${\rm Prob}({\rm rk}_p(C)=r)$
for such a $p$-group $C$, we observe that the quotient of the two
probabilities for $r=2$ and $r=1$ (for instance under the condition 
$\order C=p^2$) is at most $\frac{O(1)}{p}$ giving probabilities
in $\frac{O(1)}{p^2}$ to have $\Cl^{(p)}_{\chi*} \simeq (\Z/p\Z)^2$. 
Since ${\rm rk}_p(\Cl_{\chi})=1$ splits in the two cases of the reflection
theorem, ${\rm rk}_p(\Cl_{\chi} \oplus \Cl_{\chi*})=2$ or 
${\rm rk}_p(\Cl_{\chi*})=2$, the above applies.
As Nguyen Quang Do pointed out, this may come from the relation
${\rm H}^2 (\Cl_{\chi*}, (V/W)_{\chi*}) \simeq \F_p$, 
assuming the uniform randomness of the exact sequences
$1 \to (V/W)_{\chi*} \simeq \F_p \to \Cl^{(p)}_{\chi*} \to 
\Cl_{\chi*} \simeq \F_p \to 1$
(proof of Theorem \ref{cyclicity}), the non-cyclic case corresponding to the 
single cohomology class~$0$.

\subsubsection{Heuristics from $p$-ramification theory.}
Another investigation is about the groups ${\mathcal T}_\chi$,
$\chi \in {\mathscr X}_+$, and the formula $\order {\mathcal T}_\chi = 
\order \Cl_\chi \cdot \order {\mathcal R}_\chi$ with the equivalence \eqref{spiegel}
of reflection, $\Cl_{\chi*} \ne 1$ if and only if ${\mathcal T}_\chi  \ne 1$
(illustrated in \S\,\ref{regul}). 

\smallskip
Indeed, it is interesting to estimate in what proportions the relation 
$\order \Cl_\chi \cdot \order {\mathcal R}_\chi \ne 1$ is due to 
$\Cl_\chi$ or ${\mathcal R}_\chi$.
Of course, it is impossible to experiment with the cyclotomic fields $K$; 
so, since this problem must be considered as general and may result from 
some insights in $p$-ramification theory as done in a number 
of our articles (see \cite{Gr6} and its bibliography), we give some examples 
with quadratic and cyclic cubic fields.

\medskip
(a) {\bf Real quadratic fields and $p \geq 3$ fixed.}
For each of the ${\sf ND}$ real quadratic field of discriminant 
${\sf D} \in {\sf [bD, BD]}$, for which
${\mathcal T} \ne 1$ (counted in ${\sf Nt}$), we compute the 
proportions of cases for which this is due to $\order \Cl$ or 
$\order {\mathcal R}$; we privilegiate the case $\Cl \ne 1$ 
(counted in ${\sf Nh}$) even if the two groups $\Cl$ 
and ${\mathcal R}$ are both non-trivial; this may give a slightly 
larger proportion for $\frac{{\sf Nh}}{{\sf Nt}}$ but a much faster program:

\footnotesize
\smallskip
\begin{verbatim}
{p=3;bD=10^6;BD=10^6+5*10^4;ND=0;Nh=0;Nt=0;for(D=bD,BD,e=valuation(D,2);M=D/2^e;
if(core(M)!=M,next);if((e==1||e>3)||(e==0&Mod(M,4)!=1)||(e==2&Mod(M,4)==1),next);
m=D;if(e!=0,m=D/4);ND=ND+1;P=x^2-m;K=bnfinit(P,1);Kpn=bnrinit(K,p^2);
C5=component(Kpn,5);Hpn0=component(C5,1);Hpn=component(C5,2);
Hpn1=component(Hpn,1);vptor=valuation(Hpn0/Hpn1,p);
if(vptor>=1,Nt=Nt+1;C8=component(K,8);h=component(component(C8,1),1);
vph= valuation(h,p);if(vph>=1,Nh=Nh+1)));print("[",bD,", ",BD,"]");print
("p=",p," ND=",ND," Nt=",Nt," Nh=",Nh," Nh/Nt=",Nh/Nt+0.," 1/p=",1./p)}

[bD, BD]=[1000000, 1050000]
p=3   ND=15204   Nt=7308   Nh=2050   Nh/Nt=0.28051450  1/p=0.33333333
p=5   ND=15204   Nt=3522   Nh=634    Nh/Nt=0.18001135  1/p=0.20000000
p=7   ND=15204   Nt=2464   Nh=331    Nh/Nt=0.13433441  1/p=0.14285714
p=11  ND=15204   Nt=1497   Nh=97     Nh/Nt=0.06479625  1/p=0.09090909

[bD, BD]=[10000000, 10050000]
p=3   ND=15198   Nt=7516   Nh=2161   Nh/Nt=0.28751995  1/p=0.33333333
p=5   ND=15198   Nt=3597   Nh=720    Nh/Nt=0.20016680  1/p=0.20000000
p=7   ND=15198   Nt=2443   Nh=347    Nh/Nt=0.14203847  1/p=0.14285714
p=11  ND=15198   Nt=1512   Nh=122    Nh/Nt=0.08068783  1/p=0.09090909

[bD, BD]=[100000000, 100100000]
p=3   N=30410    Nt=15133  Nh=4456   Nh/Nt=0.29445582  1/p=0.33333333
\end{verbatim}

\normalsize
\smallskip
The proportion ${\sf Nh/Nt}$ becomes close to $\frac{1}{p}$ for intervals with large 
discriminants.

\medskip
(b) {\bf Cyclic cubic fields and $p\equiv 1 \pmod 3$ fixed.}
We obtain analogous results with the same rough calculation (e.g., we 
may have $\Cl_{\chi_1} \ne 1$ and ${\mathcal R}_{\chi_1} \ne 1$ or
${\mathcal R}_{\chi_2} \ne 1$), but this does not affect the statistics 
(${\sf f} \in {\sf [bf, Bf]}$ denotes the conductor):

\footnotesize
\smallskip
\begin{verbatim}
{p=7;bf=10^5;Bf=5*10^5;Nf=0.0;Nh=0;Nt=0;for(f=bf,Bf,e=valuation(f,3);
if(e!=0 & e!=2,next);F=f/3^e;if(Mod(F,3)!=1||core(F)!=F,next);F=factor(F);
D=component(F,1);d=component(matsize(F),1);for(j=1,d-1,l=component(D,j);
if(Mod(l,3)!=1,break));for(b=1,sqrt(4*f/27),if(e==2 & Mod(b,3)==0,next);
A=4*f-27*b^2;if(issquare(A,&a)==1,if(e==0,if(Mod(a,3)==1,a=-a);
P=x^3+x^2+(1-f)/3*x+(f*(a-3)+1)/27);if(e==2,if(Mod(a,9)==3,a=-a);
P=x^3-f/3*x-f*a/27);Nf=Nf+1;K=bnfinit(P,1);Kpn=bnrinit(K,p^2);
C5=component(Kpn,5);Hpn0=component(C5,1);Hpn=component(C5,2);
Hpn1=component(Hpn,1);vptor=valuation(Hpn0/Hpn1,p);
if(vptor>=1,Nt=Nt+1;C8=component(K,8);h=component(component(C8,1),1);
vph=valuation(h,p);if(vph>=1,Nh=Nh+1)))));print("[",bf,", ",Bf,"]");print
("p=",p," Nf=",Nf," Nt=",Nt," Nh=",Nh," Nh/Nt=",Nh/Nt," 1/p=",1./p)}

[bf, Bf]=[50000, 100000]
p=7   Nf=7928   Nt=2302   Nh=344    Nh/Nt=0.14943527   1/p=0.14285714
[bf, Bf]=[100000, 500000]
p=7   Nf=63427  Nt=18533  Nh=2690   Nh/Nt=0.14514649   1/p=0.14285714
[bf, Bf]=[100000, 500000]
p=13  Nf=63427  Nt=9979   Nh=754    Nh/Nt=0.07555867   1/p=0.07692307
[bf, Bf]=[100000, 500000]
p=19  Nf=63427  Nt=6850   Nh=389    Nh/Nt=0.05678832   1/p=0.05263157
[bf, Bf]=[100000, 500000]
p=31  Nf=63427  Nt=4316   Nh=139    Nh/Nt=0.03220574   1/p=0.03225806
\end{verbatim}

\normalsize
\smallskip
The fact that ${\mathcal R}_\chi$ is much often non-trivial than $\Cl_\chi$, 
in a computable proportion, is a positive argument for Vandiver's conjecture. 
We suggest that for totally real fields (like $K_+$), abelian $p$-ramification 
is essentially governed by the normalized $p$-adic regulator and that the 
$p$-class group is in some sense a ``secondary'' invariant.

\subsubsection{Folk heuristic.}
Consider the Gauss sum 
$\tau(\psi) = -\sm_{k=0}^{\ell-2} \zeta_p^k \cdot \xi_\ell^{g^k}$
(where $g$ is a primitive root modulo $\ell$, $\zeta_p := \psi(g)$,
see \eqref{defk}),
and put $k=a\,p+b$, $0 \leq a \leq \frac{\ell - 1}{p}-1$,  $0 \leq b \leq p-1$.
Whence:
\begin{equation}\label{combin}
\tau(\psi) = -\sm_{b=0}^{p-1} \zeta_p^b \cdot 
\big[{\rm Tr}_{\Q(\xi_\ell)/F_\ell} (\xi_\ell)\big]^{\sigma(b)}, 
\end{equation}
where $F_\ell$ is the cyclic subextension of degree $p$ of $\Q(\xi_\ell)$,
where $\sigma(b)$ is the automorphism acting
trivially on $\zeta_p$ and such that $\xi_\ell \mapsto \xi_\ell^{g^b}$,
giving an exact system of representatives for ${\rm Gal}(F_\ell/\Q)$
independent of the choice of $g$.

\smallskip
From Remark \ref{remaell}\,(ii), we know that $F_\ell$ is obtained
as the decomposition over $\Q$ of the extension $K(\sqrt[p]{\alpha})/K$, 
with $\alpha = \tau(\psi)^p \in \Z[\zeta_p]$, and it is immediate to
see that the $p$-class group of $F_\ell$ is trivial because of
Chevalley's formula on invariant classes giving here 
$\order \Cl_{F_\ell}^{{\rm Gal}(F_\ell/\Q)} =1$ since $\ell$ is the 
unique ramified prime in $F_\ell/\Q$. 

\smallskip
\quad (i) The first observation is that the $p$-class group of $F_\ell$ does 
not depend on that of $K$ as $\ell$ varies !
Indeed, this context is neither more nor less than class field theory over $\Q$
giving the existence of a unique cyclic extension $F_\ell$ of 
conductor $\ell \equiv 1 \pmod p$, for which one considers the
set of conjugates of the relative trace of $\xi_\ell$ which moreover
defines a normal basis of $F_\ell$; then the
unique link with the arithmetic of $K$ is the linear
combination \eqref{combin} involving the traces to built $\alpha$, 
but the character of $\langle \alpha \rangle_{\Z[G]} K^{\times p}/K^{\times p}$ 
is $\omega$ which gives, as we know, a ``poor'' information on the arithmetic of $K$. 

\smallskip
Thus, the relationship of $\alpha=\tau(\psi)^p$ (whence of $\tau(\psi)$) 
with class field theory over $K$ (namely, with 
$p$-classes and units of $K$) is tenuous, possibly empty;
which is quite the opposite for the twists ${\rm g}_c(\ell)$ because 
of the relations $\alpha^{c - s_c} = {\rm g}_c(\ell)^p$ and the fact that
the ${\rm g}_c(\ell)_{\chi^*}$ are radicals defining non-trivial (arithmetically)
cyclic extensions of degree $p$ of $K_+$ for any even character $\chi$.

\smallskip
\quad (ii) In another direction, suggested by the work of Lecouturier \cite{Le}
generalizing results of Calegari--Emerton and Iimura, consider the non-Galois 
extension $\wt F_\ell := \Q(\sqrt[p]{\wt \alpha})$, where $\wt \alpha := \ell$; 
of course, $K(\sqrt[p]{\wt \alpha})/K$ is 
a cyclic extension of degree $p$ (undecomposed over a strict subfield
of~$K$), ramified at the $p-1$ primes ${\mathfrak L} \mid \ell$ and at $p$
if and only if $\ell \not\equiv 1 \pmod {p^2}$. 
On the contrary, as shown by many results of \cite{Le},
the $p$-class group of $\wt F_\ell$ strongly depends on 
the arithmetic of $K$ while the radical $\wt \alpha$ does not.

\smallskip
This second observation comes from the fact that, for
$\wt M:=K(\sqrt[p]{\wt \alpha})$:
$$\order \Cl_{\wt M}^{{\rm Gal}(\wt M/K)} =
\order \Cl_K \cdot \frac{p^{p-2+\delta}}
{(E_K : E_K \cap {\rm N}_{\wt M/K} (\wt M^\times))} \leq
\order \Cl_K \cdot p^{\frac{p-1}{2}}, $$

where $\delta = 1$ or $0$
according as $p$ ramifies or not and where $\zeta_p$ is norm 
for $\delta=0$; but the non-abelian Galois 
structure yields various non-trivial $p$-class groups for 
$\wt F_\ell$ as $\ell$ varies, and genera theory implies 
${\rm rk}_p(\Cl_{\wt F_\ell}) \geq 1$ for all $\ell$ (for the
metabelian genera theory, see \cite{J}).
However, for $M = K(\sqrt[p]{\alpha}) = F_\ell \, K$:
$$\order \Cl_{M}^{{\rm Gal}(M/K)} = \order \Cl_K \cdot \frac{p^{p-2}}
{(E_K : E_K \cap {\rm N}_{M/K} (M^\times))}
\leq \order \Cl_K \cdot p^{\frac{p-1}{2}} ,$$ 

and we left to the reader the computation of the (non-trivial) order of the minus part;
but $M/K$ decomposes into $F_\ell/\Q$ and only the 
isotopic component for the unit character is concerned, which gives 
in fact a trivial part of the above Chevalley's formula (contrary to 
the metabelian case $\wt M/\Q$). 
So the ``folk heuristic'' should be:

\smallskip
\begin{pushright}
{\it Because of $F_\ell$ defined by the radical $\alpha = \tau(\psi)^p$,
the $p$-adic properties of the Gauss sums are independent 
of the arithmetic of $K$ as $\ell$ varies (despite 
the apparent complexity of the radical $\alpha = \tau(\psi)^p$), 
while the properties of $\wt F_\ell$ are strongly dependent 
(despite the obvious simplicity of the radical $\wt \alpha= \ell$).}
\end{pushright}

\smallskip
In other words we have probably some ``dualities'' about the arithmetic 
complexity of Kummer theory in the comparison ``radicals versus 
extensions''.

\subsection{Additive $p$-adic statistics}\label{add}
Of course, we are only concerned with the multiplicative $p$-adic
properties of the Gauss sums $\tau(\psi)$ and mainly of the twists
${\rm g}_c(\ell)$, and these are given by their
$\chi^*$-components for $\chi \in  {\mathscr X}_+$. 
Nevertheless, the additive properties
seem to follow more explicit rules, which is an interesting
information about the numerical repartition and the independence 
as $\ell$ varies, and this probably has an impact on the multiplicative 
properties regarding the results of \S\,\ref{classes}.
We shall examine the case of the twists ${\rm g}_c(\ell)$ (more
precisely of $\psi^{-c}(c)\,{\rm g}_c(\ell)$), then the case of the original 
Gauss sums $\tau(\psi)$ from the arithmetic of the fields $F_\ell$.

\subsubsection{$\Z$-rank of the family 
$\big(\psi^{-c}(c)\,{\rm g}_c(\ell) \big)_{\ell \in {\mathscr L}_p}$.} 
Put, for $p$ and $c$ fixed:
\begin{equation}\label{jel}
{\bf J}(\ell) := \psi^{-c}(c)\,{\rm g}_c(\ell) =
\psi^{-c}(c)\, \tau(\psi)^{c-\sigma_c} 
= J_1 \cdots J_{c-1} \ \hbox{(see \eqref{jacobi})}
\end{equation} 

written on the basis $\{1, \zeta_p, \ldots, \zeta_p^{p-2}\}$,
under the form
${\bf J}(\ell) = \sm_{k=0}^{p-2} a_k(\ell)\,\zeta_p^k$,
the integers $a_k(\ell)$ being considered modulo $p$. 
A first information, about the $p$-adic 
repartition of the ${\bf J}(\ell)$ as $\ell$ varies, is to compute 
the $\F_p$-rank of the $\F_p$-matrix $\big(a_k(\ell) \big)_{k,\ell}$. 
The following  program gives {\it systematically}:
$${\rm Rank}_{\F_p} \big[\big(a_k(\ell) \big)_{\ell, k} \big] = p-4, $$

for all the primes $p \geq 7$ tested 
(rank $1$ for $p=3$ and rank $2$ for $p=5$), {\it despite the 
fact that the lines are not canonical} (up to circular permutations 
of their elements since ${\bf J}(\ell)$ is defined up to conjugation). 
We have verified it up to $p \leq 331$, 
an interval which contains $16$ irregular primes.
The program gives $p$, the $\F_p$-rank 
of the matrix (in ${\sf rank}$) and the least 
$\ell_p$ (in ${\sf elp}$) for which the sub-matrix built from 
$\{\ell \in {\mathscr L}_p, \ \ell \leq \ell_p\}$ has rank $p-4$:

\footnotesize
\smallskip
\begin{verbatim}
{forprime(p=7,500,c=lift(znprimroot(p));P=polcyclo(p)+Mod(0,p);M=matrix(0,p-1);
r=0;for(i=1,10^8,el=1+2*i*p;if(isprime(el)!=1,next);g=znprimroot(el);J=1;
for(i=1,c-1,Ji=0;for(k=1,el-2,kk=znlog(1-g^k,g);e=lift(Mod(kk+i*k,p));
Ji=Ji-x^e);J=J*Ji);J=lift(Mod(J,P));V=vector(p-1,j,component(J,j));
A=concat(M,V);rr=matrank(A);if(rr==r,next);r=rr;M=A;
if(r==p-4,print("p=",p," r=",r," elp=",el);break)))}
p  rank  elp       p   rank  elp       p   rank   elp       p   rank   elp
7   3    113       11   7    397       13   9     599       17   13    1259   
(...)
71  67   42743     73   69   48473     79   75    50087     83   79    65239  
151 147  247943    157  153  273181    163  159   294053    167  163   305611
\end{verbatim}

\normalsize
\smallskip
We have ${\bf J}(\ell) \equiv 1 \pmod {\mathfrak p}$, in other words
$\sum_{k=0}^{p-2} a_k(\ell) \equiv 1 \pmod p$, and we can write
${\bf J}(\ell) = 1+\sum_{k=1}^{p-2} a_k(\ell)\,(\zeta_p^k-1)$ 
depending on $p-2$ parameters; then, due to the relations
${\bf J}(\ell)^{1+s_{-1}} \equiv 1 \pmod p$ and 
${\bf J}(\ell)^{e_\omega} \in K^{\times p}$ (because 
$\omega(c - s_c) \equiv 0 \pmod p$), this yields the
three relations of  ``derivation'' (for $p \geq 7$)
$\sm_{k=1}^{p-2} k^\delta \! \cdot a_k(\ell) \equiv 0 \pmod p$, 
$\delta \in \{1,2,4\}$, for any $\ell \in {\mathscr L}_p$.
Whence a $\F_p$-rank at most $p-4$, but we have no proof of 
the equality. 

\smallskip
The order of magnitude of $\ell_p$ seems to be $O(1)\,p^2\,{\rm log}(p^2)$, 
which is in agreement with a ``conductor'' $p^2$ for these Hecke 
Gr\"ossencharacters \cite[Theorem, p.\,489]{We}, but the program slows 
down very much, as $p$ increases, to be more accurate. 

\smallskip
Moreover, the number of consecutive primes $\ell$ needed to reach the 
rank $p-4$ is equal to $p-4$, except probably for finitely many cases,
which confirms the above order.
Give now the end of the above table with an estimation of the $O(1)$:

\footnotesize
\smallskip
\begin{verbatim}
p    elp   O(1)    p    elp    O(1)    p    elp   O(1)    p    elp   O(1)    
211 517373 1.0856  223 628861  1.1693  227 604729 1.0816  229 631583 1.1082
233 642149 1.0849  239 695491  1.1116  241 684923 1.0750  251 784627 1.1269
257 862493 1.1766  263 819509  1.0631  269 928051 1.1461  271 906767 1.1019
277 925181 1.0719  281 1055437 1.1853  283 979747 1.0834  293 988583 1.0136
307 1174583 1.0881 311 1214767 1.0941 313 1203799 1.0692 317 1276243 1.1026
\end{verbatim}

\normalsize
\smallskip
The $\F_p$-rank $r_p(\ell)$ of the $p-1$ conjugates of ${\bf J}(\ell) \pmod p$,
$\ell \in {\mathscr L}_p$, is close to $p-4$ (e.g., for $p=37$, 
$r_{37}(\ell) \in \{33, 32, 31, 30\}$ in similar proportions, and we only have the 
local minimum $(r_{37}(\ell), \ell)=(29,2591)$ for $\ell$ up to $37000$.

\subsubsection{Repartition of the conjugates of the traces 
${\rm Tr}_{\Q(\xi_\ell)/F_\ell} (\xi_\ell)$.}
Let  $Z_{F_\ell}$ be the ring of integers of $F_\ell$ and let
$Z_{F_\ell}/ p\,Z_{F_\ell}$ be the 
residue ring modulo~$p$. These residue rings are 
isomorphic to $\F_{p^p}$ or to $\F_p^p$, but there is no canonical 
map between them as $\ell \in {\mathscr L}_p$ varies. 

Thus, in the expression \eqref{combin} giving
$\tau(\psi) = -\sm_{b=0}^{p-1} \zeta_p^b \cdot 
\big[{\rm Tr}_{\Q(\xi_\ell)/F_\ell} (\xi_\ell)\big]^{\sigma(b)}$,
the images in $Z_{F_\ell}/ p\,Z_{F_\ell}$ of the conjugates of the
relative traces ${\rm Tr} (\xi_\ell) := {\rm Tr}_{\Q(\xi_\ell)/F_\ell} (\xi_\ell)$
may be easily analysed and compared, for $\ell \in {\mathscr L}_p$, 
by means of the image $R_\ell$ in $\F_p[x]$ of the polynomial
$Q_\ell=\prd_{\ov \sigma \in {\rm Gal}(F_\ell/\Q)} 
\big(x-{\rm Tr} (\xi_\ell)^{\ov \sigma} \big) \in \Z[x]$.

\begin{proposition} 
Let $\ell_1 , \ell_2 \in {\mathscr L}_p$ and let
$\tau(\psi_1)$, $\tau(\psi_2)$ be the corresponding 
Gauss sums normalized via
$\psi_1(g_1) = \psi_2(g_2)=\zeta_p$.
Let $F = F_{\ell_1}F_{\ell_2}$.\par

If $R_{\ell_1} \ne R_{\ell_2}$, then for all $\sigma \in {\rm Gal}(FK/\Q)$,
$\tau(\psi_2) \not \equiv \tau(\psi_1)^{\sigma}\! \pmod {{\mathfrak p}^pZ_{FK}}$.
\end{proposition}

\begin{proof}
Suppose there exists $\sigma \in {\rm Gal}(FK/\Q)$ such that 
$\tau(\psi_2) \equiv \tau(\psi_1)^{\sigma} \pmod {{\mathfrak p}^pZ_{FK}}$;
recall that $\tau(\psi_1)^{\sigma} = \zeta_\sigma \,\tau(\psi_1^e)$,
$\zeta_\sigma \in \mu_p$, $e \in (\Z/p\Z)^\times$. Then:

$\tau(\psi_2) = -\sm_{b=0}^{p-1} \zeta_p^b \cdot 
{\rm Tr}(\xi_{\ell_2})^{\sigma_2(b)}\ $ and
$\ \tau(\psi_1)^{\sigma} = -\sm_{b=0}^{p-1} \zeta_p^b 
\cdot {\rm Tr}(\xi_{\ell_1})^{\pi(\sigma_1(b))}$,\par

where $\pi$ is a permutation of the $\sigma_1(b)$.
Using ${\rm Tr}_{\Q(\xi_{\ell_i})/\Q}(\xi_{\ell_i})=-1$, we get:\par
$\tau(\psi_2) = 1 - \sm_{b=1}^{p-1} (\zeta_p^b - 1) \cdot 
{\rm Tr} (\xi_{\ell_2})^{\sigma_2(b)}$, 
$\tau(\psi_1)^{\sigma} = 1 - \sm_{b=1}^{p-1} (\zeta_p^b - 1) \cdot 
{\rm Tr} (\xi_{\ell_1})^{\pi(\sigma_1(b))}$, 
whence:\par
$\tau(\psi_1)^{\sigma} \!- \tau(\psi_2)  =\! \sm_{b=1}^{p-1} 
(\zeta_p^b -1) \! \cdot \! \big ({\rm Tr} (\xi_{\ell_2})^{\sigma_2(b)} \! - 
{\rm Tr} (\xi_{\ell_1})^{\pi(\sigma_1(b))}\big ) 
\!\equiv 0\!\! \pmod{{\mathfrak p}^pZ_{FK}}$.

Since the $\zeta_p^b -1$, $b \in [1, p-1]$, define a $\Z$-basis of 
${\mathfrak p}\,Z_K$, then a $Z_F$-basis of ${\mathfrak p}\,Z_{FK}$,
this relation implies ${\rm Tr} (\xi_{\ell_2})^{\sigma(b)}
\equiv {\rm Tr} (\xi_{\ell_1})^{\pi'(\sigma(b))} \pmod p$
for all $b$, which yields $R_{\ell_1}=R_{\ell_2}$ in $\F_p[x]$ 
(contradiction).
\end{proof}

Since $\tau(\psi_2) \not \equiv \tau(\psi_1)^{\sigma} 
\pmod {{\mathfrak p}^p}$ for all $\sigma$ implies 
${\rm g}_c(\ell_2) \not \equiv {\rm g}_c(\ell_2)^\sigma 
\pmod {{\mathfrak p}^p}$ for all $\sigma$ 
(except for the $\omega$-components 
because $\omega(c - \sigma_c) \equiv 0 \pmod p$), 
we can say that the number of distinct polynomials 
$R_\ell$, $\ell \in {\mathscr L}_p$, gives a partial 
idea of the repartition modulo $p$ of the 
sets ${\mathscr E}_\ell(p)$ as $\ell$ varies.
 As $p$ increases, the number of distinct $R_\ell$ seems to be
$O(p^2 \cdot {\rm log}(p^2))$.

\smallskip
The following program, computing the monic polynomial 
${\sf R} = R_\ell \in \F_p[x]$ returns: ${\sf el} = \ell$, the residue 
degree ${\sf f} = f$ of $p$ in $F_\ell/\Q$, and ${\sf R}$.

\footnotesize
\smallskip
\begin{verbatim}
{p=7;B=5*10^3;el=1;while(el<B,el=el+2*p;if(isprime(el)!=1,next);g=znprimroot(el);
h=g^p;g=lift(g);h=lift(h);P=polcyclo(el);z=Mod(x,P);Q=1;e=1;for(k=1,p,Tr=0;e=e*g;
for(j=1,(el-1)/p,e=e*h;e=lift(Mod(e,el));Tr=Tr+z^e);Q=Q*(T-Tr));
Q=component(lift(Q),1);R=0; for(i=0,p,C=component(Q,i+1);C=lift(Mod(C,p));
R=R+x^i*C);F=znorder(Mod(p,el));f=1;v=valuation(F,p);w=valuation(el-1,p);
if(w==v,f=p);print("el=",el," f=",f," R=",R))}

el=29    f=7     R=x^7 + x^6 + 2*x^5 + 5*x + 1
el=43    f=1     R=x^7 + x^6 + 3*x^5 + 3*x^3 + 6*x^2
el=71    f=7     R=x^7 + x^6 + 5*x^5 + 3*x^4 + 2*x^3 + 6*x^2 + 4
el=4943  f=7     R=x^7 + x^6 + 3*x^5 + x^4 + x^3 + 3*x + 5
el=4957  f=7     R=x^7 + x^6 + 4*x^5 + 2*x^4 + 5*x^3 + 3*x^2 + 2*x + 1
el=4999  f=7     R=x^7 + x^6 + 4*x^3 + 5*x^2 + 2*x + 6
\end{verbatim}

\normalsize
\smallskip
It is hopeless to write wide lists of polynomials $R_\ell$ for 
large $p$, but any experiment suggests a random distribution 
of the (non-independent) coefficients (except that of $x^{p-1}$ 
since ${\rm Tr}_{\Q(\xi_\ell)/\Q}(\xi_\ell) =-1$).
For $p=3$ the six possible polynomials are of the form $R_\ell$.
For $p=5$ (resp. $p=7$) there are $150$ (resp. $17192$) possible polynomials.

\smallskip
\quad (i) For $p=5$, we obtain the following end of the calculations 
(two days of computer; it seems that only $35$ distinct polynomials 
$R_\ell$ are available):

\footnotesize
\smallskip
\begin{verbatim}
el=5591   f=5     R=x^5 + x^4 + 4*x^3 + x^2 + 4*x + 2
el=6211   f=1     R=x^5 + x^4 + x^3 + x^2 + x
el= 6271  f=1     R=x^5 + x^4 + 2*x^3 + 4*x^2 + 3*x + 4
el=1345   f=1     R=x^5 + x^4
\end{verbatim}

\normalsize
\smallskip
\quad (ii) For $p=7$, $\ell$ up to $17977$, we get painfully a little 
more than $250$ distinct $R_\ell$, but the exact number is unknown.

\begin{remark} It is clear that a large number of polynomials $R_\ell$ 
strengthens Vandiver's conjecture since the corresponding 
${\bf J}(\ell) = \psi^{-c}(c)\,{\rm g}_c(\ell)$ cover sufficiently possibilities modulo $p$, 
especially since  we know that the $\F_p$-rank associated to the family of 
$\big({\bf J}(\ell)\big)_{\ell \in {\mathscr L}_p}$ is probably always
$p-4$, but these informations are not ``equivalent''. 
Moreover, an assumption about the order of magnitude of 
${\mathscr N}_p := \order \{R_\ell, \ \ell \in {\mathscr L}_p\}$
{\it is not necessary} to obtain Vandiver's conjecture for $p$;
indeed, a {\it single} suitable $\ell$ may ensure a positive test for 
Vandiver's conjecture as shown by the table given in \S\,\ref{exist}.
\end{remark}

We propose the following heuristic, about the sets
${\mathscr E}_\ell(p)$ of exponents of $p$-primarity
for which the reference \cite{GZ} may be usefull:

\smallskip
\begin{pushright}
{\it The probability of 
${\mathscr E}_\ell(p) = \ev$, for a single $\ell \in {\mathscr L}_p$, 
is $(1+o(1)) \cdot e^{-\frac{1}{2}}$; that of at least a 
counterexample to Vandiver's conjecture is of the form
$O(1)\,\big (1 - e^{-\frac{1}{2}} \big)^{{\mathscr N}_p}$, 
where ${\mathscr N}_p := \order \{R_\ell, \ \ell \in {\mathscr L}_p\}$,
with the polynomial $\,R_\ell=\prod_{\ov \sigma \in {\rm Gal}(F_\ell/\Q)} 
\big(x-{\rm Tr}_{\Q(\xi_\ell)/F_\ell} (\xi_\ell)^{\ov \sigma} \big)$
seen in $\F_p[x]$.}
\end{pushright}

\section{Conclusion}
Under these experiments and heuristics, the existence of sets 
${\mathscr E}_\ell(p)$, disjoint from ${\mathscr E}_0(p)$, or 
probably the existence of primes $\ell \in {\mathscr L}_p$ such 
that ${\mathscr E}_\ell(p) = \ev$, may occur conjecturally for all $p$. 
Possibly, our computations in \S\,\eqref{exist} show the existence of 
general properties of the sets ${\mathscr E}_\ell(p)$ coming from the fact 
that {\it all $\ell \in {\mathscr L}_p$ intervene} (and that these primes are probably 
independent), which is a new argument compared with classical ones.
This is strengthened by the computation of the conjugates of the traces 
${\rm Tr}_{\Q(\xi_\ell)/F_\ell} (\xi_\ell)$, as $\ell \in {\mathscr L}_p$ varies
(coefficients of the Gauss sums), the fields $\Q(\mu_\ell)$ 
being, a priori, independent of the arithmetic of $K$.

\begin{remark} 
There are two constraints, for the Gauss and Jacobi  sums that we have 
considered, but they only concern the auxiliary prime numbers 
$\ell \in {\mathscr L}_p$:

\smallskip
\quad (i) The $p$-classes of ideals 
${\mathfrak L} \mid \ell$, $\ell \in {\mathscr L}_p$, are all 
represented with standard densities.

\smallskip
\quad (ii) The ideal factorization of $\tau(\psi)^p$ is related to 
{\it congruences modulo the conjugates of a prime ideal 
${\mathfrak L} \mid \ell$ and is canonical} (this yields 
Stickelberger's theorem and its consequences 
\cite[\S\,15.1]{Wa}, \cite{C,We}, and \cite{Sch} for the 
annihilation of $\Cl_{\chi_0^*}^{(p)}$ with generalizations 
of the Stickelberger ideal).
A similar context is that of the $\ell$-adic $\Gamma$-function of Morita.

\smallskip
However, since we consider characters $\psi$ of order $p$,
the {\it $p$-adic congruential properties} of Gauss sums 
(or Jacobi sums) do not follow any known law (in our opinion, the
classical literature being mute about this).
\end{remark}

These fundamental $p$-adic properties of Gauss sums
may have crucial consequences in various domains since
Vandiver's conjecture is often required; for instance:

\smallskip
In \cite{DP} about the Galois cohomology of Fermat curves, in
\cite{Shu} for the root numbers of the Jacobian varieties of Fermat 
curves, then in several papers on Galois $p$-ramification 
theory as in \cite{McCS,Sha0,Sha1,Sha2}, or \cite{WE1,WE2} in 
relation with modular forms, then in numerous papers and books 
on the theory of deformations of Galois representations 
as in \cite{Be,Me}, Iwasawa's theory context and cyclotomy, 
as in \cite{Col} on Ihara series, \cite{BP} for $\mu$-invariants 
in Hida families, \cite{KW} for the main conjecture of the Iwasawa theory).

\smallskip
Then it may be legitimate to think that all these numerous 
basic congruential aspects are (logically) governing principles 
of a wide part of algebraic number theory, as follows, beyond 
the case of the $p$th cyclotomic field (not to mention all the 
geome\-trical aspects as the theory of elliptic curves where some
analogies can be found, and all the generalizations of the present 
abelian case over a number field $k \ne \Q$):

\smallskip
\begin{pushright}
{\it Gauss and Jacobi sums $\,\too \,$ Hecke Gr\"ossencharacters $\,\too \,$ 
Stickelberger element $\,\too \,$  
$p$-adic $L$-functions $\,\too \,$
Herbrand--Ribet theorem $\,\too \,$ Main Theorem on abe\-lian fields $\,\too \,$ 
annihilation of the $p$-torsion group ${\mathcal T}$ of real 
abelian fields $\,\too \,$ universal isomorphism ${\mathcal T} \simeq 
{\rm H}^2(G_{S_p},\Z_p)^*$  $\,\too \,$ $p$-rationality of fields 
(${\mathcal T} =1$) $\,\too \,$ cohomological obstructions in Galois theory 
$\, \too \ \cdots$}
\end{pushright}

\smallskip 
Which gives again an example of a {\it basic $p$-adic problem}, 
analogous to those we have analysed about deep conjectures: 
Greenberg's conjectures (on Iwasawa theory over totally real fields 
\cite{Gre1} and on representation theory \cite{Gre3}), $p$-rationalities 
of a number field as $p\to\infty$, generalizations of the conjecture of
Ankeny--Artin--Chowla from the conjectural existence of a $p$-adic 
Brauer--Siegel theorem \cite{Gr6}\,$\ldots$

\smallskip
As shown by the evidences given in \S\,\ref{new}, Vandiver's conjecture 
may be justified, for $p \gg 0$, by the Borel--Cantelli 
heuristic, on exceptional features of Gauss sums; but this point of view 
allows cases of failure of the conjecture, which is not satisfactory 
for the theoretical foundations of the above quoted fundamental subjects.

\smallskip
To be optimistic (but not very rigorous), one can 
say that Vandiver's conjecture is true because it holds 
for sufficiently many prime numbers \cite{BH,HHO} since 
probabilities may be in $\ds\frac{O(1)}{p^{\lambda(p)}}$, 
$\lambda(p) \to \infty$. 
In a more serious claim, we can say that Vandiver's conjecture 
holds for almost all prime numbers; the accurate cardinality of the finite set 
of counterexamples ($\ev$ or not) is (in our opinion) not of algebraic nature 
nor enlightened by class field theory, Galois cohomology or Iwasawa's theory, 
but is perhaps accessible by the way of analytical/geometrical techniques or 
depends on a more general hypothetic ``complexity theory'' in number theory.


\end{document}